\documentclass[12pt, twoside]{article}
\usepackage{amsmath,amsthm,amssymb,mathtools,mathrsfs}
\usepackage{times, hyperref}
\usepackage{enumerate, enumitem}

\usepackage[T1]{fontenc}        % Codage des fontes
\usepackage[utf8]{inputenc}   % Codage du fichier
\usepackage[francais]{babel}

\def\titlerunning#1{\gdef\titrun{#1}}
\makeatletter
\def\author#1{\gdef\autrun{\def\and{\unskip, }#1}\gdef\@author{#1}}
\def\address#1{{\def\and{\\\hspace*{18pt}}\renewcommand{\thefootnote}{}%
\footnote {#1}}%
\markboth{\autrun}{\titrun}}
\makeatother
\def\email#1{e-mail: #1}
\def\subjclass#1{{\renewcommand{\thefootnote}{}%
\footnote{\emph{Mathematics Subject Classification (2010):} #1}}}
\def\keywords#1{\par\medskip
\noindent\textbf{Keywords.} #1}

\theoremstyle{plain}
\newtheorem{theoreme}{Théorème}[section]
\newtheorem{corollaire}[theoreme]{Corollaire}
\newtheorem{lemme}[theoreme]{Lemme}
\newtheorem{proposition}[theoreme]{Proposition}

\theoremstyle{definition}
\newtheorem*{merci}{Remerciements}

\theoremstyle{remark}
\newtheorem*{remarque}{Remarque}
\newtheorem*{remarques}{Remarques}

\newcommand\numberthis{\stepcounter{equation}\tag{\theequation}}

\let\Re\relax
\DeclareMathOperator{\cond}{cond}
\DeclareMathOperator{\Re}{Re}

\renewcommand{\bar}{\overline}
\newcommand{\ssum}[1]{\sum_{\substack{#1}}}
\newcommand{\ssumm}[1]{\underset{\substack{#1}}{\sum\sum}}
\newcommand{\e}{{\rm e}}
\newcommand{\dd}{{\rm d}}
\newcommand{\ca}{{\mathfrak a}}
\newcommand{\cb}{{\mathfrak b}}
\renewcommand{\hat}{\widehat}

\newcommand{\ee}{{\varepsilon}}
\newcommand{\rhoQ}{{\rho}}
\newcommand{\rhod}{{\varrho}}

\newcommand{\NN}{{\mathbb N}}
\newcommand{\CC}{{\mathbb C}}
\newcommand{\RR}{{\mathbb R}}
\newcommand{\ZZ}{{\mathbb Z}}
\newcommand{\PP}{{\mathbb P}}

\newcommand{\1}{{\mathbf 1}}

\newcommand{\cC}{{\mathcal C}}
\newcommand{\cD}{{\mathcal D}}

\newcommand{\sQ}{{\mathscr Q}}
\newcommand{\cI}{{\mathcal I}}
\newcommand{\cL}{{\mathcal L}}
\newcommand{\cP}{{\mathscr P}}
\newcommand{\cS}{{\mathcal S}}

\newcommand{\cB}{{\mathscr B}}
\newcommand{\cM}{{\mathscr M}}
\newcommand{\cE}{{\mathscr E}}
\newcommand{\cH}{{\mathscr H}}

\newcommand{\reg}{{\text{rég}}}
\newcommand{\exc}{{\text{exc}}}
\newcommand{\Mrg}{\cM^\reg}
\newcommand{\Mex}{\cM^\exc}

\newcommand{\ud}{{\frac12}}

\newcommand{\tsigma}{{\tilde\sigma}}

\newcommand{\hV}{{\widehat V}}

\newcommand{\tQ}{{\tilde Q}}

\newcommand{\mo}{\mathopen}
\newcommand{\mc}{\mathclose}

\DeclareMathOperator{\card}{card}
\DeclareMathOperator{\supp}{supp}

\newcommand{\sumb}[1]{\mathop{\underset{#1}{\sum}^{\flat}}}

\newcommand{\vphi}{{\varphi}}
\renewcommand{\tilde}{\widetilde}

\newcommand{\avoidbreak}{\postdisplaypenalty=100}

\renewcommand{\mod}[1]{\ ({\rm mod\ }{#1})}

\numberwithin{equation}{section}

\begin{document}

%%%%% To ease editing, add:

\baselineskip=17pt

%%%%%%%%%%%%%%%%

%% In the running head, give an abbreviation of the title. 
\titlerunning{Équirépartition des polynômes quadratiques et valeurs friables de polynômes}

\title{Niveau de répartition des polynômes quadratiques \\ et crible majorant pour les entiers friables}
% \title{Plus grand facteur premier des polynômes quadratiques et majoration du nombre de valeurs friables de polynômes}

\author{R. de la Bretèche \and S. Drappeau}

\date{\today}

\maketitle

\address{R. de la Bretèche: Institut de Mathématiques de Jussieu-Paris Rive Gauche, 
Université  Paris Diderot, Sorbonne Paris Cité, UMR 7586, Bâtiment Sophie Germain, 
Case Postale 7012, F-75251 Paris CEDEX 13, France; \email{regis.delabreteche@imj-prg.fr}
\and
S. Drappeau: Aix Marseille Université, CNRS, Centrale Marseille, I2M UMR 7373, 13453 Marseille, France; \email{sary-aurelien.drappeau@univ-amu.fr}}

\subjclass{Primary 11N32; Secondary 11B25, 11L05, 11N75}

%%%%%%%%

\begin{abstract}
We obtain new estimates on the level of distribution of the set~$\{Q(n)\}$ where~$Q\in\ZZ[X]$ is irreducible quadratic, for well-factorable moduli, improving a result due to Iwaniec. As a by-product of our arguments, we study the Chebyshev problem of estimating~$\max\{P^+(n^2-D), n\leq x\}$ and make explicit, in Deshouillers-Iwaniec's state-of-the-art result, the dependence on the Selberg eigenvalue conjecture.

Combined with the construction of an upper-bound sieve for numbers free of large factors, we obtain new upper bounds for the quantity~$\Psi_Q(x, y) = |\{n\leq x: p\mid Q(n)\Rightarrow p\leq y\}|$ for~$Q\in\ZZ[X]$ linear or quadratic.
%We apply this to upper-bound the density of integers divisible by the square of their largest prime factor ( De Koninck-Doyon-Lucas (numbers divisible by the square, and Goudout.
%This is used to substantially improve the known upper-bound on the density of integers divisible by the square of their largest prime factor.

%% Keywords are optional
\keywords{polynomial values, quadratic congruences, arithmetic progressions, friable numbers}
\end{abstract}

\section{Introduction}

\subsection{Niveau de répartition de suites polynomiales et problème de Tchebychev}

Soit~$Q$ un polynôme à coefficients entiers. La question d'estimer les ``sommes de congruence''~$A_d(x) := \left|\{n\in \ZZ\cap[1,x], Q(n)\equiv 0\mod{d}\}\right|$, lorsque~$x\to+\infty$, le plus uniformément possible par rapport à l'entier~$d\geq 1$, est au cœur de la théorie multiplicative des nombres. Toute réponse partielle à cette question permet, en conjonction avec les méthodes de cribles~\cite{HR,FI-OC}, d'approcher la fréquence avec laquelle~$Q$ prend des valeurs sous une contraine multiplicative~: par exemples des valeurs premières~\cite[Theorem~2.6]{HR}, ou ayant un grand facteur premier~\cite{Tenenbaum-ES2}, ou encore n'ayant que des petits facteurs premiers, cas sur lequel nous nous concentrerons ci-dessous.

Dans le présent travail, nous nous intéressons à ce problème lorsque~$Q$ est quadratique, dans le cas particulier important des modules~$d$ pondérés par un poids ``bien factorisable'' au sens de la théorie du crible linéaire, qui a notamment été étudié dans~\cite{Iwaniec}.

\begin{theoreme}\label{th:exporep-quad-intro}
Soient~$\eta>0$,~$x\geq 1$, $Q\leq x^{1+25/178-\eta}$, et~$(\lambda(q))$ une fonction arithmétique bornée et ``bien factorisable'' au sens de~\cite[page~199]{FoIw}. Soit~$D\in\ZZ$ qui n'est pas un carré d'entier, et~$V:\RR\to\CC$ une fonction lisse à support compact inclus dans~$\RR_+^\ast$. Alors
\begin{equation}
\begin{aligned}
\sum_{q\leq Q} \lambda(q) \Big(\ssum{n\in\NN \\ q|n^2-D} V\Big(\frac nx\Big) - x{\hat V}(0) \frac{\rhoQ(q)}{q}\Big) \ll_{\eta, V, D} x^{1-\eta/3}
\end{aligned} \label{eq:exporep-quad-intro}
\end{equation}
avec~${\hat V}(0) = \int_\RR V(t)\dd t$ et~$\rhoQ(q)=|\{\nu\mod{q}, \nu^2\equiv D\mod{q}\}|$.
\end{theoreme}

Le gain~$25/178 = 1/7,\!12$ améliore le gain correspondant~$1/15$ de~\cite{Iwaniec} (voir aussi~\cite{RJLO}), ainsi que le gain~$1/9$ qui découlerait d'une utilisation de la conjecture~$R^\ast$ de Hooley sur les sommes d'exponentielles incomplètes (\textit{cf.}~\cite[page~185]{Iwaniec}). Pour obtenir cette amélioration, nous faisons appel à des majorations de type ``grand crible'' sur les coefficients de Fourier des formes cuspidales de~$GL_2$, adaptant un argument de T\'o{th}~\cite{Toth}, en utilisant une récente extension~\cite{D-Kuz} par le second auteur des majorations de sommes d'exponentielles de Deshouillers-Iwaniec~\cite{DI}.

Les méthodes qui sous-tendent le Théorème~\ref{th:exporep-quad-intro} s'appliquent naturellement au problème de Tchebychev de minorer la fonction~$P_D(x) = P^+(\prod_{x<n\leq 2x}(n^2-D))$, pour l'historique duquel nous référons à~\cite{Hooley-Pplus} (voir aussi~\cite{Dartyge}). Le dernier résultat en date sur cette question précise, du à Deshouillers et Iwaniec~\cite{DI-Pplus}, implique en particulier que
$$ P_D(x) \geq x^{1,2024} $$
pour tout~$x$ suffisamment grand. D'un autre côté, la conjecture de Selberg sur les valeurs propres du Laplacien hyperbolique (\textit{cf.} \cite{DI}, section 1.3) permettrait d'obtenir, pour tout~$\ee>0$ suffisamment petit, la valeur $\sqrt{3/2}-\ee \geq 1,2247$ en exposant. Nos résultats permettent de rendre explicite, dans les arguments de Deshouillers-Iwaniec~\cite{DI-Pplus}, la dépendance vis-à-vis de la conjecture de Selberg.
\begin{theoreme}\label{th:DI-P}
Pour~$\theta\in[0, 1/4]$, définissons~$\kappa(\theta)\in[1, 2]$ comme l'unique réel satisfaisant
$$ \int_1^{\kappa(\theta)} \frac{t\dd t}{1-2\theta t} = \frac1{4(1-2\theta)}. $$
Pour tout~$\ee>0$ et~$D\in\ZZ$ qui n'est pas un carré d'entier, nous avons
$$ P_D(x) \gg_{\ee, D} x^{\kappa(\theta)-\ee} $$
pour tout~$\theta\geq0$ qui est admissible pour la conjecture de Ramanujan-Selberg. En particulier,~$\theta=7/64$ convient~\cite{Kim}; nous avons donc pour~$x$ suffisamment grand
\begin{equation}
P_D(x) \geq x^{1,2182}.\label{eq:Pplus-incond}
\end{equation}
\end{theoreme}

\subsection{Crible majorant pour les entiers friables}

Notre application principale qui motive le Théorème~\ref{th:exporep-quad-intro} concerne la majoration de la fréquence avec laquelle~$Q$ prend des valeurs sans grand facteur premier. Nous ne supposerons plus nécessairement que~$Q$ est quadratique. Dans le présent travail, nous améliorons les résultats connus dans le cas où~$Q$ est linéaire ou produit de deux facteurs linéaires, et nous obtenons les premiers résultats non triviaux sur cette question lorsque~$Q$ est quadratique irréductible.

On dit qu'un entier $n$ est $y$-friable si, et seulement si, son plus grand facteur premier noté $P^+(n)$ est $\leq y$. Nous adopterons la convention $P^+(1)=1$.  
Les travaux d'Hildebrand et Tenenbaum~\cite{HT} complétés par ceux de Saias~\cite{Saias} permettent d'évaluer asymptotiquement $\Psi(x,y):=\card\{ S(x,y)\}$ où $S(x,y):=\{ n\leq x\,:\, P^+(n)\leq y\}$ dans un large domaine en $y$.

Nous souhaitons estimer la densité des entiers $y$-friables dans la suite $\{ Q(n)\}$. Malheureusement seul le cas des polynômes de degré $1$ a été pour l'instant complètement résolu. Nous pourrons nous reporter à~\cite{FT,LB98} pour l'uniformité en fonction des coefficients.
Dans le cas du degré $\geq 2$, même une majoration du bon ordre de grandeur semble intéressant et permettrait d'importantes applications.

Lorsque  $Q$ est un polynôme de $\ZZ[X]$, posons $$\Psi_Q(x,y):=\card\{ n\leq x\,:\, P^+(Q(n))\leq y\}.$$
Puisque l'on peut trivialement se ramener au cas des polynômes sans facteur carré, nous factorisons
\begin{equation}\label{factorisationQ} 
àQ(X):=\prod_{j=1}^r Q_j(X)
\end{equation}
où $Q_j$ sont des polynômes irréductibles non proportionnels de $\ZZ[X]$ de degré $d_j$. Martin~\cite{Martin} a conjecturé l'équivalent asymptotique
$$\Psi_Q(x,y)\sim   x\prod_{j=1}^r \frac{\Psi (x^{d_j},y)}{x^{d_j}}$$ dans un large domaine en $y$.
En désignant par $\rhod$ la fonction de Dickmann et en utilisant~\cite{Hild86},  lorsque $y\geq \exp\{ (\log\log x)^{3/5+\varepsilon}\} $, nous espérons donc
$$\Psi_Q(x,y)\sim   x\prod_{j=1}^r \rhod(d_ju)$$
où $u:=(\log x)/(\log y)$. 

Lorsque~$(d_1,r)\neq (1,1)$, cette relation n'a été établie que certains cas, et lorsque~$u$ est restreint à de petits intervalles bornés. Nous renvoyons à la section 4.3 du survol~\cite{Granville}.

C'est donc à l'aune de cette prévision que l'on pourra mesurer la qualité des majorations que nous établirons.
Ainsi, nous conjecturons 
\begin{equation}\label{eq:psiQ-conj}
\Psi_Q(x,y)=x\rho(u)^{\sum_{j=1}^r d_j+o(1)}
\end{equation}
lorsque $u$ tend vers l'infini dans un large domaine en $y$.

Lorsque $Q$ est un polynôme en plusieurs variables, il est possible d'obtenir certaines estimations asymptotiques et nous renvoyons le lecteur intéressé aux articles de~\cite{BBDT,Lachand2, Lachand3}. Dans le cas d'un polynôme univarié, qui nous intéresse ici, des minorations du bon ordre de grandeur~$\Psi_Q(x, y)\gg x$ sont établies par Dartyge, Martin et Tenenbaum~\cite{DMT} pour des petites valeurs de~$u$. Dans le cas particulier d'un produit de deux formes linéaires, Hildebrand~\cite{Hild-Balogs} obtient une minoration du bon ordre de grandeur; nous renvoyons à~\cite{BET} pour une minoration effective. Enfin, lorsque~$Q$ est un produit quelconque de formes linéaires, une minoration qualitative mais dépendant explicitement de~$u$ est obtenue dans~\cite{BW}.

L'objectif du présent travail est de montrer des majorations de~$\Psi_Q(x, y)$ qui approchent aussi près que possible la taille conjecturée~\eqref{eq:psiQ-conj}. Une conséquence de nos résultats simple à énoncer  est la suivante.

\begin{theoreme}\label{th1simple} Soit $Q$ un polynôme de $\ZZ[X]$ fixé avec la factorisation \eqref{factorisationQ} et $d_1\leq d_2\leq \cdots\leq d_r$. Soit~$\psi:\RR_+\to\RR_+$ telle que~$\lim_{x\to+\infty}\psi(x)=+\infty$. Nous avons
$$\Psi_Q(x,y)\ll x\rhod(u)^{c_Q+o(1)}$$
lorsque~$x$ tend vers l'infini,~$(\log x)^{\psi(x)} \leq y\leq x$ et
$$ \def\arraystretch{1.2}
\begin{array}{rll}
c_Q = {}& \kern-0.7em \tfrac85  & \qquad (d_1=d_2=1,\, r\geq 2), \\
c_Q = {}& \kern-0.7em \tfrac32  & \qquad (d_1=1,\, d_2\geq 2), \\
c_Q = {}& \kern-0.7em 1+\tfrac{25}{178} & \qquad (d_1=2) \\
c_Q = {}& \kern-0.7em 1 & \qquad (d_1\geq 3,~\text{\cite{Khmyrova}}).
\end{array} $$
\end{theoreme}
Nous notons que~$\frac{25}{178} = \frac1{7,\!12} \approx 0,\!1404$. Le dernier cas permet de retrouver un résultat de Khmyrova~\cite{Khmyrova}, précisé ultérieurement par Timofeev~\cite{Timofeev}. L'estimation du Théorème~\ref{th1simple} dans le cas~$d_1 = 2$ est, à notre connaissance, la première faisant intervenir une valeur~$c_Q>1$.

Chacune de ces majorations repose sur une inégalité essentiellement du type
\begin{equation}
1_{P^+(n)\leq y}(n)\leq \sum_{\substack{d\mid n\\ P^+(d)\leq y\\ R/y<d\leq R }}1\label{eq:majo-basique-crible}
\end{equation}
pour tout $R\geq n$. Le choix du paramètre $R$ doit être optimisé en fonction des résultats d'équirépartition à notre disposition. Dans le premier cas, il s'agira d'une rapide adaptation des récents travaux du second auteur, dans le deuxième, du théorème de Harper~\cite{Harper} de type Bombieri--Vinogradov pour les entiers friables. Dans le cas quadratique, nous utiliserons une variante du Théorème~\ref{th:exporep-quad-intro}. Dans le dernier cas, nous choisirons $R=x$ sans pouvoir utiliser de résultats spécifiques d'équirépartition.

Notons aussi que d'après~\cite{GY-Sunits} (voir également~\cite{Schinzel}), un polynôme~$Q$ avec la factorisation~\eqref{factorisationQ} et de degré au moins~$2$, ne peut pas prendre des valeurs trop friables~: en effet, il découle du corollaire~4 de~\cite{GY-Sunits} que
$P^+(Q(n))\gg ( \log_2n)(\log_3n)/\log_4n$
lorsque $n$ tend vers l'infini\footnote{Ici $\log_k$ désigne la $k$-ième itérée du logarithme.}. Ainsi, par exemple, la $y$-friabilité de $n$ et celle de $n+1$ ne sont pas indépendantes lorsque $y$ prend des très petites valeurs par rapport à la taille de $n$.

À titre d'illustration de l'efficacité de notre méthode dans le cas de deux facteurs linéaires, nous appliquons ces estimations à l'étude de l'ensemble~$\cC$ des entiers divisibles par le carré de leur plus grand facteur premier,
$$ \cC = \{n\geq 1:\ P^+(n)^2 | n\}. $$
Il est aisé d'établir que lorsque~$x$ tend vers l'infini,~$|\cC \cap [1, x]|  = x\e^{-(1+o(1))\sqrt{2(\log x)\log\log x}}. $ %= \sum_p \Psi\Big(\frac{x}{p^2}, p\Big)
L'ensemble des entiers~$n$ tels que~$(n, n+1)\in\cC^2$ est en revanche beaucoup plus délicat à étudier. Nous établissons la majoration suivante, qui améliore~\cite{DKDL}.
\begin{corollaire}\label{cor:dkdl}
Lorsque~$x$ tend vers l'infini, l'estimation
\begin{equation}
|\{n\leq x:\ (n, n+1)\in\cC^2\}| \leq x  \e^{-(c+o(1))\sqrt{2(\log x)\log\log x}}\label{eq:majo-DKDL}
\end{equation}
a lieu avec~$c=4/\sqrt{5} \approx 1,\!789 $.
\end{corollaire}

\begin{remarque}
Dans~\cite{DKDL}, les auteurs obtiennent~$c=25/24 \approx 1,\!042$. Conjecturellement, la valeur optimale attendue est~$c=2$. Cela correspond à l'heuristique que les événements~$n\in\cC$ et~$n+1\in\cC$ surviennent de façon statistiquement indépendante.
\end{remarque}

Une autre application concerne la minoration du nombre d'entiers friables dans les petits intervalles établie au Théorème 5 de \cite{G16}. Lorsque $\varepsilon\in \mo{]}0,1/6\mc{[}$, il existe $u_0=u_0(\varepsilon)$ tel que lorsque  $x\geq 2 $, $u_0\leq u\leq (\log x)^{1/6-\varepsilon}$ et $ \rhod(u)^{-(3+\varepsilon)}\leq h\leq \sqrt{x}$, l'on ait
$$ \Psi(x+h\sqrt{x},x^{1/u})-\Psi(x ,x^{1/u})\geq \rhod(u)^2\frac{h\sqrt{x}}{(\log x)^3}. $$
La méthode repose de manière cruciale sur des majorations de cardinal d'ensembles d'entiers~$n$ tels que les valeurs de deux formes affines en~$n$ soient simultanément friables. Les majorations obtenues au Théorème 1.1 avec un exposant $c_Q = 8/5$ dans le cas $r=2$, $d_1=d_2=1$ permettent d'obtenir des résultats intéressants à ce sujet et de remplacer l'exposant $3+ \varepsilon$ par $3/2+\varepsilon$ dans la minoration de $h$ ci-dessus. La valeur conjecturelle $c_Q=2$ fournirait un exposant~$1/2+ \varepsilon$.

\begin{merci}
Les auteurs prennent plaisir à remercier Adam Harper et Gérald Tenenbaum pour des discussions sur le présent travail.
\end{merci}

\subsection*{Notations et plan}

Lorsque~$(a, q)=1$, nous notons~$\bar{a}\mod{q}$ la classe inverse de~$a$ modulo~$q$. Le conducteur d'un caractère de Dirichlet~$\chi$ est noté~$\cond\chi$. Étant donnés~$m, n\in\NN_{>0}$, nous notons~$m \mid n^\infty$ lorsque tous les facteurs premiers de~$m$ divisent~$n$. Ainsi, nous notons~$(m, n^\infty)$ le plus grand diviseur de~$m$ dont tous les facteurs premiers divisent~$n$. La lettre~$\rhod$ désigne la fonction de Dickman, tandis que la notation~$\rhoQ_Q(q)$ pour un polynôme~$Q$ et un entier~$q$ désigne le nombre de racines de~$Q$ modulo~$q$. Nous noterons~$\rho_Q(q)=\rho(q)$ lorsque le contexte sera clair. Dans les section~\ref{sec:kuz}, \ref{sec:final-hfix} et~\ref{sec:final-hmoy} uniquement, la notation~$(\rho_{f,\ca}(n))$ désigne la suite des coefficients de Fourier d'une forme de Maass~$f$. Pour~$k\geq 1$, nous définissons la fonction diviseur généralisée~$\tau_k$ par~$\sum_n \tau_k(n) n^{-s} = \zeta(s)^k$ pour tout~$s>1$. Nous utiliserons l'abbréviation~$\e(x) = \e^{2\pi i x}$ ($x\in\RR$).

Pour tout~$u\geq 1$, nous notons
$$ H(u) := \exp\Big\{\frac{u}{(\log 2u)^2}\Big\}. $$

Pour un ensemble~$\cI\subset\RR_+$, nous notons
$$ \Psi(\cI, y ; q, a) := |\{n\in\cI, P^+(n)\leq y, n\equiv a\mod{q}\}|, $$
$$ \Psi_q(\cI, y) := |\{n\in\cI, P^+(n)\leq y, (n, q)=1\}|, $$
$$ \Psi(\cI, y; \chi) = \sum_{n\in\cI, \ P^+(n)\leq y} \chi(n). $$
% Il sera également utile de noter, pour~$\lambda\in\RR$,
% \begin{equation}
% \psi_\lambda(n) := \prod_{p|n}\Big(1+\frac1{p^\lambda}\Big).\label{eq:def-psia}
% \end{equation}
% Nous observons que~$n/\vphi(n)\ll_\ee \psi_{1-\ee}(n)$, et
% $$ \sum_{n\leq x} \psi_{\ee}(n) \ll_\ee x \qquad (\ee>0, x\geq 1) $$
Nous utiliserons à plusieurs reprises l'estimation
\begin{equation}
\Psi(x, y) = x\rhod(u) \e^{O_\ee(u)} \qquad ((\log x)^{1+\ee} \leq y \leq x),\label{eq:compar-psi-rho-unif}
\end{equation}
qui découle directement des formules~(2.6) et~(2.7) de~\cite{HT}.

\bigskip{}

Le plan de ce travail est le suivant. Dans la section 2, nous établissons des résultats d'équirépartition des entiers friables en progressions arithmétiques de module friable. Dans la section 3, nous définissons une certaine fonction~$\vartheta(q)$, qui jouera dans notre contexte le rôle d'un crible majorant pour les entiers friables. Dans la section 4, nous considérons le cas~$d_1=1$, où nous mettons en jeu les résultats de la section 2. Dans la section 5, nous considérons le cas quadratique~$d_1=2$, en supposant acquis le Théorème~\ref{th:exporep-quad-intro}. Dans la section 6, nous étudions le cas~$d_1\geq3$. Dans la section 7, nous déduisons le Corollaire~\ref{cor:dkdl}. Enfin, dans la section 8 nous prouvons le Théorème~\ref{th:exporep-quad-intro}.

\section{Équirépartition en progressions arithmétiques}

Les travaux de Harper et du second auteur permettent de démontrer une approximation de~$\Psi(\cI, y_1 ; q, a)$ par~$\Psi_q(\cI, y_1)/\vphi(q)$ en moyenne sur les modules~$q$. En vue de nos applications, nous montrons dans cette section une version de ces résultats qui est restreinte aux modules~$q$ qui sont~$y_2$-friables.

\begin{theoreme}\label{th:equidistrib}
Soient~$\ee, A>0$ et~$k\geq 1$. Il existe~$C, \delta>0$ tels que pour~$x^\ee \leq Q\leq x^{3/5-\ee}$ et $(x, \eta, y_1, y_2)\in\RR^4$ avec~$(\log x)^C \leq y_1, y_2 \leq x$, $y_1\leq x^{1/C}$ et $\eta\in[\e^{-\delta\sqrt{\log y}}, 1]$, et pour tout entier non nul~$q_0\leq x^\delta$ avec~$(q_0, a_1a_2)=1$ et~$P^+(q_0)\leq y_2$, l'on ait
\begin{equation}\label{eq:equidistrib-35}
\begin{aligned}
\ssum{Q/2 < q \leq Q \\ P^+(q) \leq y_2 \\ (q, a_1a_2)=1} \tau_k(q) \Big|\Psi([x, {}& (1+\eta)x], y_1 ; q_0q, a_1\bar{a_2}) - \frac{\Psi_{q_0q}([x, (1+\eta)x], y_1)}{\vphi(q_0q)}\Big| \\ &\ll_{\ee, A, k}  \eta \Psi(x, y_1) \frac{\Psi(Q, y_2)}{\vphi(q_0)Q} \frac{\e^{O_k(u_2)}}{(\log x)^A}
\end{aligned}
\end{equation}
uniformément pour~$0< |a_1|, |a_2| \leq x^\delta$, où l'on a posé~$u_2 = (\log x)/\log y_2$.

De plus, lorsque~$Q\leq x^{1/2-\ee}$, l'on a
\begin{equation}
\label{eq:equidistrib-12}
\begin{aligned}
\ssum{Q/2 < q \leq Q \\ P^+(q) \leq y_2} \tau_k(q)\max_{(a, q_0q)=1}{}& \Big|\Psi([x, (1+\eta)x], y_1 ; q_0q, a) - \frac{\Psi_{q_0q}([x, (1+\eta)x], y_1)}{\vphi(q_0q)}\Big| \\ &\ll_{\ee, A, k}  \eta \Psi(x, y_1) \frac{\Psi(Q, y_2)}{\vphi(q_0)Q} \frac{\e^{O_k(u_2)}}{(\log x)^A}
\end{aligned}
\end{equation}
\end{theoreme}

\begin{proof}
Les arguments de~\cite{Harper}, notamment l'inégalité~(3.1) et les calculs qui la suivent, montrent que pour prouver la borne annoncée, il suffit de montrer que
\begin{equation}\label{eq:somme-chi}
\begin{aligned}
\ssum{Q/2 < q \leq Q \\ P^+(q) \leq y_2}\frac{\tau_k(q)}{\vphi(q_0q)}\ssum{\chi\mod{q_0q}\\1<\cond \chi \leq x^\delta}{}& \big|\Psi([x, (1+\eta)x], y_1 ; \chi)\big| \\ &
\ll_{\ee, A, k} \frac{\eta \Psi(x, y_1) \Psi(Q, y_2)\e^{O_k(u_2)} }{\vphi(q_0)Q (\log x)^A}.
\end{aligned}
\end{equation}
où~$\chi$ désigne un caractère de Dirichlet, et~$\cond\chi$ est le conducteur de~$\chi$. La section 3.3 de l'article de Harper~\cite{Harper} traite du cas~$y_2 = Q$ et~$\eta=1$. Les changements nécessaires pour tenir compte des conditions~$P^+(q)\leq y_2$ sont mentionnés dans~\cite{LV, D-fr}. Nous nous contentons donc d'indiquer que le point essentiel est que pour tout~$r\leq x^\delta$,
\begin{equation}\label{eq:majo-q}
\ssum{Q/2 < q \leq Q \\ P^+(q)\leq y_2 \\ r|q} \frac{\tau_k(q)}{\vphi(q)} \ll (\e^{u_2}\log x)^{O_k(1)}\frac{\tau_k(r)r^{1-\alpha_2}}{\vphi(r)}\frac{\Psi(Q, y_2)}Q.
\end{equation}
Ici,~$\alpha_2 = \alpha(x, y_2)$ et~$u_2 = (\log x)/\log y_2$. Lorsque~$C$ est choisi suffisamment grand, par l'inégalité triangulaire et les calculs de la section~3.3 de~\cite{Harper}, nous obtenons une borne
$$ \ll_\ee \frac{\Psi(x, y_1)\Psi(Q, y_2)}{\vphi(q_0)Q}(\e^{u_2}\log x)^{O_k(1)}\big\{y_1^{-2\delta} + \e^{-2\delta\sqrt{\log x}}\big\} $$
pour la contribution à la somme~\eqref{eq:somme-chi} provenant des caractères dont la série de Dirichlet n'a pas de zéro ``de Siegel'', c'est-à-dire trop proche de~$1$ (voir~\cite{Harper}, section~3.3, définition de~${\mathcal G}_1$ pour la définition exacte). Étant donnée l'hypothèse~$\eta\geq \e^{-\delta\sqrt{\log y}}$, nous concluons que ces caractères fournissent une contribution acceptable.

Il reste alors à traiter la contribution~$\cS_1$ au membre de gauche de~\eqref{eq:somme-chi} des caractères éventuels de conducteurs au plus~$\min\{y^\eta, \e^{\eta\sqrt{\log x}}\}$ dont la série de Dirichlet possède un zéro réel supérieur à~$1-M/\min\{\log y, \sqrt{\log x}\}$ pour des constantes absolues adéquates~$\eta$ et~$M$. Si de tels caractères existent, ils sont tous induits par un unique caractère primitif~$\chi_1$, de conducteur~$r_1\leq \min\{y_1^b, \e^{b\sqrt{\log x}}\}$ (notés respectivement~$\chi_{\text{bad}}^*$ et~$r_{\text{bad}}$ dans~\cite{Harper}), pour un certain réel~$b>0$~:
\begin{equation}
\label{eq:def-S1}
\cS_1 := \ssum{Q/2 < q \leq Q \\ P^+(q) \leq y_2,\ r_1 |q}\frac{\tau_k(q)}{\vphi(q_0q)}\ssum{\chi\mod{q_0q}\\ \chi_1 \text{ induit } \chi } \big|\Psi([x, (1+\eta)x], y_1 ; \chi)\big|.
\end{equation}

Tout d'abord, nous observons que l'inégalité triangulaire permet de nous ramener trivialement au contexte de la section 3.3 de~\cite{Harper}, ce qui fournit la borne
\begin{equation}
\big|\Psi([x, (1+\eta)x], y_1 ; \chi)\big| \ll_\ee \Psi(x, y_1) (\log x)^2\label{eq:majo-Psi-chi1}
\end{equation}
pour~$C$ suffisamment grand en fonction de~$\ee$. Cette borne découle directement de la première équation, page~16 de~\cite{Harper} dans le cas~$x^{1/(\log\log x)^2}\leq y$. Dans le cas complémentaire, nous utilisons l'inégalité de la dernière équation, page~17 de~\cite{Harper} ainsi que
\begin{equation}
\sum_{p|q_0}p^{-\alpha} \leq \sum_{p\leq p_\ast} p^{-\alpha} \ll \log\log p_\ast + p_\ast^{1-\alpha} \qquad (\pi(p_\ast)= \omega(q_0)),\label{eq:majo-euler-p-alpha}
\end{equation}
et la majoration élémentaire~$p_\ast\ll \log q_0 \ll \log x$, qui implique~$p_\ast^{1-\alpha} \ll y_1^{(1-\alpha)/3} \ll u_1^{1/2}$. Les majorations~\eqref{eq:majo-q}, \eqref{eq:majo-Psi-chi1} et la minoration~$r_1\gg_A (\log x)^{10A}$ provenant du théorème de Siegel~\cite[théorème~II.8.33]{ITAN} fournissent
\begin{equation}
\cS_1 \ll_{\ee, A, k} \Psi(x, y_1) \frac{\Psi(Q, y_2)}{\vphi(q_0)Q} \frac{\e^{O_k(u_2)}}{r_1^{1/2} H(u_1)^\delta (\log x)^A }. \label{eq:S1-Harper}
\end{equation}
Cette borne est acceptable si~$r_1 > \e^{\delta\sqrt{\log y_1}}$, ou bien si~$H(u_1)> \e^{\sqrt{\log y_1}}$.

Nous pouvons donc supposer dans ce qui suit que~$r_1\leq \e^{\delta\sqrt{\log y_1}}$ et~$H(u_1)\leq \e^{\sqrt{\log y_1}}$. En particulier, nous avons~$y_1 \gg_\ee \e^{(\log x)^{2/3-\ee}}$.

Nous reprenons les arguments de la section~6 de~\cite{FT}, et en particulier les équations~(6.29) et~(6.36). Celles-là fournissent
$$ \Psi([x, (1+\eta)x], y_1 ; \chi_1) \ll |F(\log((1+\eta)x)) - F(\log x)| + \Psi(x, y_1) \e^{-2\delta\sqrt{\log y_1}}, $$
avec
$$ F(t) := \e^{\beta t} \int_0^\infty y_1^{-\beta v}\big\{\omega(\tfrac{t}{\log y_1}-v)-\e^{-\gamma}\big\}K(y_1^v)\dd v, \qquad K(t) := \sum_{n\leq t} \chi_1(n), $$
où~$\omega$ désigne la fonction de Buchstab~\cite[équation~(6.5)]{FT} et~$\beta$ désigne le zéro exceptionnel, proche de~$1$, de~$L(s, \chi_1)$. Pour~$\e^t \not\in y_1\NN$, nous avons
$$ F'(t) = \beta F(t) + \e^{\beta t} \int_0^\infty y_1^{-\beta v} \omega'\big(\tfrac{t}{\log y_1}-v\big)K(y_1^v)\dd v - y_1^\beta K(\e^t/y_1) $$
où~$\omega'$ est prolongée par continuité à droite aux points où~$\omega$ est non dérivable. En utilisant le lemme~6.2 de~\cite{FT}, l'inégalité de Pólya-Vinogradov~\cite[théorème~II.8.10]{ITAN} et le fait que~$\eta\leq 1$, il vient
\begin{align*}
& |F(\log((1+\eta)x)) - F(\log x)| \\ \ll{}& \eta \big( x^\beta \int_0^\infty y_1^{-\beta v} \rhod(u_1-v) H(u_1)^{-\delta}\e^{2\delta v} |K(y_1^v)|\dd v + y_1^\beta \sqrt{r_1}\log r_1\big) \\
\ll{}& \eta x^\beta \rhod(u_1)
\end{align*}
où nous avons utilisé l'hypothèse~$y_1\leq x^{1/C}$ pour majorer le terme~$y_1^\beta \sqrt{r_1}\log r_1$. Nous déduisons
$$ \Psi([x, (1+\eta)x], y_1 ; \chi_1) \ll \Psi(x, y_1)\big\{\eta x^{\beta-1} \log x + \e^{-\delta \sqrt{\log y_1}}\big\}. $$
Nous insérons cela dans~\eqref{eq:def-S1}, en utilisant encore~\eqref{eq:majo-q} et~\eqref{eq:majo-euler-p-alpha}. Nous obtenons
\begin{equation}
\cS_1 \ll_{\ee, A, k} \Psi(x, y_1)\frac{\Psi(Q, y_2)}{\vphi(q_0)Q}\e^{O_k(u_2)}\big\{\eta(\log x)^{-A} + \e^{-\delta\sqrt{\log y_1}}\big\}.\label{eq:S1-FT}
\end{equation}
En réinterprétant le paramètre~$\delta$, la conjonction des deux bornes~\eqref{eq:S1-Harper} et~\eqref{eq:S1-FT} implique donc la borne annoncée
$$ \cS_1 \ll_{\ee, A, k} \eta \Psi(x, y_1) \frac{\Psi(Q, y_2)}{\vphi(q_0) Q}\frac{\e^{O_k(u_2)}}{(\log x)^A}. $$
\end{proof}

\section{Crible majorant pour les entiers friables}

La fonction de~$n$ au membre de droite de l'inégalité~\eqref{eq:majo-basique-crible} fonctionne de façon analogue à un crible, en ce qu'elle majore la fonction indicatrice d'un ensemble d'entiers au moyen d'informations sur les diviseurs de petite taille des éléments de cet ensemble. Dans cette section nous établissons la version précise de l'inégalité~\eqref{eq:majo-basique-crible} que nous utiliserons.

\begin{proposition}\label{prop:poids}
Pour tous~$\ee, \kappa>0$, il existe~$u_0>0$ tel que lorsque~$(R, y)\in\RR^2$, $(\log R)^{2+\ee} \leq y\leq R^{1/u_0}$ et~$D\in\NN$, il existe une fonction~$\vartheta = \vartheta_{R, y, D, \kappa}:\NN \to \RR$ ayant les propriétés suivantes.
\begin{enumerate}[label=(\roman*)]
\item Nous avons~$|\vartheta(q)| \leq \tau_3(q)$ pour tout~$q\in\NN$.
\item Nous avons~$\vartheta(q)=0$ si~$q\not\in [R, y^4R]$, $(q, D)>1$ ou~$P^+(q)>y$.
\item Pour tout~$n\in\NN$, nous avons
$$ \1_{P^+(n)\leq y} \leq \sum_{q|n} \vartheta(q) $$
si~$n/(n, D^\infty) \geq R$, et~$\sum_{q|n}\vartheta(q) \geq 0$ dans tous les cas.
\item Pour~$m\in\NN$, et~$f$ une fonction multiplicative avec~$0\leq f(p^\nu)\leq \kappa$, nous avons
$$ \Big|\sum_{(q, m)=1} \frac{\vartheta(q)f(q)}{q}\Big| \leq \e^{O_{\ee, \kappa}((\log R)/\log y)}\rhod(\tfrac{\log R}{\log y}). $$
La constante implicite dépend au plus de~$\ee$ et~$\kappa$.
\end{enumerate}
\end{proposition}

\begin{proof}
Soit~$m := n/(n, D^\infty)$, supposons~$m\geq R$ et écrivons~$m=p_1 \dotsb p_J$ où les nombres~$p_j$ sont premiers et~$p_j \leq p_{j+1}$ pour tout~$j$. En raisonnant comme dans~\cite{Khmyrova, Vaughan}, considérons la suite~$(c_j)_0^J$ de diviseurs de~$m$ définie par~$c_0=1$ et, pour~$j\geq 1$, $c_j = p_1 \dotsb p_j$. Puisque par hypothèse~$c_J = m\geq R>1$, nous avons nécessairement~$c_{j-1}<R\leq c_j$ pour un certain~$j\geq 1$. Notons~$d_0$ l'entier~$c_j$ ainsi formé. Il s'agit d'un diviseur de~$n$, premier avec~$D$, satisfaisant les conditions~$R\leq d_0 < RP^+(d_0)$ et~$P^-(m/d_0)\geq P^+(d_0)$. Nous avons donc
\begin{equation}
\1_{P^+(n)\leq y} \leq \ssum{R\leq d_0 < RP^+(d_0) \\ P^+(d_0)\leq y \\ d_0|n, (d_0, D)=1} \1_{P^-(m/d_0)\geq P^+(d_0)}.\label{eq:premajo-HV}
\end{equation}
Cette majoration implique bien~\eqref{eq:majo-basique-crible}. Il est cependant important, pour les grandes valeurs de~$y$, de ne pas oublier la condition~$P^-(m/d_0)\geq P^+(d_0)$, au risque d'engendrer un facteur supplémentaire de l'ordre de~$\log y$ dans nos applications. Pour contourner cela, pour chaque entier~$d_0$ apparaissant dans~\eqref{eq:premajo-HV}, nous introduisons des coefficients de crible majorant~$(\xi_{d_1})$ de niveau~$y^2$ et dimension~$\kappa$, pour les nombres premiers~$p< P^+(d_0)$, en suivant la construction donnée par exemple dans~\cite[Lemma~5]{FI} où~$s$ et~$z$ sont remplacés respectivement par~$2$ et~$P^+(d_0)$. En particulier, on a~$\xi_1=1$, $|\xi_{d_1}|\leq 1$, et
$$ \xi_{d_1} \neq 0 \Rightarrow \begin{cases} d_1\leq y^2, \\ P^+(d_1)<P^+(d_0), \\ \mu^2(d_1)=1. \end{cases} $$
Enfin, ces coefficients vérifient la majoration uniforme (\cite[lemme~5]{FI})
$$ \sum_{(d_1, m)=1} \frac{\xi_{d_1}f(d_1)}{d_1} \ll_\kappa \prod_{\substack{p < P^+(d_0) \\ p\nmid m}} \Big(1+\frac{f(p)}p\Big)^{-1} \qquad (m\in\NN), $$
lorsque la fonction~$f$ est multiplicative avec~$0\leq f(p^\nu) \leq \kappa$. Notons que ces propriétés impliquent bien
$$ \1_{P^-(m/d_0)\geq P^+(d_0)} \leq \ssum{d_1|n/d_0 \\ (d_1, d_0D)=1} \xi_{d_1}. $$
Nous insérons cette majoration dans~\eqref{eq:premajo-HV}. En posant, pour tout~$q\in\NN$,
\begin{equation}
\vartheta(q) := \underset{\substack{q=d_0d_1 \\ R \leq d_0 < RP^+(d_0) \\ P^+(d_0)\leq y \\ (d_0d_1, D)=(d_0, d_1)=1}}{\sum \sum}\ \xi_{d_1},\label{eq:def-theta}
\end{equation}
il résulte de ce qui précède que
$$ \1_{P^+(n)\leq y} \leq \ssum{R\leq d_0 < RP^+(d_0) \\ P^+(d_0)\leq y \\ d_0|n, (d_0, D)=1} \ssum{d_1|n/d_0 \\ (d_1, d_0D)=1} \xi_{d_1} = \sum_{q|n} \vartheta(q). $$

Il reste à montrer la propriété~(iv) qui porte sur une majoration en moyenne. Nous avons
\begin{align*}
S :={}& \sum_{(q, m)=1} \frac{\vartheta(q)f(q)}{q} \\={}& \ssum{R\leq d_0 < RP^+(d_0) \\ (d_0, m)=1} \frac{f(d_0)}{d_0} \ssum{(d_1, md_0D)=1} \frac{\xi_{d_1}f(d_1)}{d_1}  \\ \ll_\kappa {}& \ssum{R\leq d_0 < RP^+(d_0) \\ (d_0, m)=1} \frac{f(d_0)}{d_0}\prod_{\substack{p<P^+(d_0) \\ p\nmid md_0 D}}\Big(1+\frac{f(p)}p\Big)^{-1}.\end{align*}
Nous séparons la somme sur~$d_0$ suivant~$\ell=P^+(d_0)$, $\ell^\nu \| d_0$. Nous obtenons
$$ S \ll_\kappa \ssum{\ell\text{ premier} \\ \ell \leq y \\ \ell^{\nu-1}\leq R} \frac{f(\ell^\nu)}{\ell^\nu} \prod_{\substack{p<\ell \\ p\nmid mD}} \Big(1 + \frac{f(p)}p\Big)^{-1} \ssum{R/\ell^\nu \leq d < R/\ell^{\nu-1} \\ P^+(d)<\ell \\ (d, mD)=1 } \frac{f(d)}{d}\prod_{p|d}\Big(1+\frac{f(p)}p\Big), $$
Notons~$\beta_\ell = \alpha(R/\ell, \ell)$ si~$\ell\geq (\log R)^{2+\ee}$, et~$\beta_\ell=\tfrac{1+\ee}{2+\ee}>\tfrac12$ sinon. La somme sur~$d$ ci-dessus est
$$ \leq \Big(\frac{R}{\ell^\nu}\Big)^{-1+\beta_\ell} \ssum{P^+(d)<\ell \\ (d, mD)=1} \frac{f(d)}{d^{\beta_\ell}}\prod_{p|d}\Big(1+\frac{f(p)}p\Big) \ll_{\ee,\kappa} \Big(\frac{R}{\ell^\nu}\Big)^{-1+\beta_\ell} \prod_{\substack{p<\ell \\ p\nmid mD}}\Big(1+\frac{f(p)}{p^{\beta_\ell}}\Big) $$
où nous avons utilisé le fait que~$\beta_\ell>\min\{\tfrac34, \tfrac12+\tfrac\ee{10}\}$ lorsque~$u_0$ (donc~$R$) est suffisamment grand. Nous avons donc
\begin{equation}
\begin{aligned}
S {}&\ll \ssum{\ell \text{ premier} \\ \ell \leq y \\ \ell^{\nu-1} \leq R} \Big(\frac{R}{\ell^\nu}\Big)^{-1+\beta_\ell}\frac{1}{\ell^\nu}\prod_{\substack{p<\ell \\ p\nmid mD}}\Big(1+O_\kappa\Big(\frac{p^{1-\beta_\ell}-1}p\Big)\Big) \\
{}&\ll \ssum{\ell \text{ premier} \\ \ell \leq y} \Big(\frac{R}{\ell}\Big)^{-1+\beta_\ell}\frac{1}{\ell}\prod_{\substack{p<\ell \\ p\nmid mD}}\Big(1+O_\kappa\Big(\frac{p^{1-\beta_\ell}-1}p\Big)\Big). \label{eq:S-majo-prefinal}
\end{aligned}
\end{equation}
Le produit eulérien est borné par
\begin{equation}
\exp\Big\{O_\kappa\Big(\sum_{p<\ell}\frac{p^{1-\beta_\ell}-1}p\Big)\Big\} = \exp\Big\{O_\kappa\Big(\min\Big\{\frac{\log R}{\log \ell}, \frac{\ell^{\frac1{2+\ee}}}{\log \ell}\Big\}\Big)\Big\}.\label{eq:majo-somme-prem}
\end{equation}
en utilisant~\cite[lemme~3.6]{BT05}. Nous insérons cette majoration dans le membre de droite de~\eqref{eq:S-majo-prefinal}, et notons~$S_1+S_2$ la somme résultante, avec~$S_1$ la contribution de~$\ell\leq (\log R)^{2+\ee}$ et~$S_2$ la contribution complémentaire. Ainsi, nous avons
$$ S_1 \ll R^{-\frac1{2+\ee}} \ssum{\ell\text{ premier} \\ \ell\leq (\log R)^{2+\ee}}\frac{\e^{O_\kappa(\ell^{1/(2+\ee)}/\log \ell)}}{\ell^{(1+\ee)/(2+\ee)}} \ll R^{-\frac1{2+\ee} +O_\kappa(1/\log\log R)}. $$
Considérons ensuite~$S_2$, en notant~$v_\ell := (\log R)/\log \ell$ et~$u := (\log R)/\log y$. Nous avons ainsi
$$ S_2 \ll \ssum{\ell\text{ premier} \\ (\log R)^{2+\ee} < \ell\leq y} \frac{\e^{O_\kappa(v_\ell)}}{\ell (R/\ell)^{1-\beta_\ell}}. $$
Nous utilisons ensuite l'estimation, déduite de la formule~(7.8) de~\cite{HT}
$$ (R/\ell)^{\beta_\ell-1} = \e^{O_\kappa(v_\ell)} \frac{\Psi(R/\ell, \ell)}{R/\ell}. $$
En découpant la somme en intervalles~$y^{1/(j+1)}<\ell \leq y^{1/j}$, nous obtenons
\begin{align*}
\ssum{\ell\text{ premier} \\ (\log R)^{2+\ee} < \ell\leq y} \frac{\e^{O_\kappa(v_\ell)}}{\ell (R/\ell)^{1-\beta_\ell}} {}& \leq \frac1R \ssum{j\geq 1 \\ (\log R)^{2j} \leq y} \e^{O_\kappa(ju)} \ssum{\ell\text{ premier} \\ \ell\leq y^{1/j}} \Psi(R/\ell, \ell) \\ {}& = \frac1R \ssum{j\geq 1 \\ (\log R)^{2j} \leq y} \e^{O_\kappa(ju)} \Psi(R, y^{1/j})
\end{align*}
par l'identité de Buchstab~\cite{Buchstab}. Nous utilisons ensuite l'estimation~\eqref{eq:compar-psi-rho-unif}, qui fournit
$$ \Psi(R, y^{1/j}) \ll \e^{O_\kappa(u)}\Big(\frac{\e^{O_\kappa(1)}}{u\log u}\Big)^{(j-1)u} \Psi(R, y), $$
pour conclure que~$S_2 \ll \e^{O_\kappa(u)}\frac{\Psi(R, y)}{R}$ lorsque~$u_0$ est suffisamment grand. Lorsque~$y\geq (\log R)^{2+2\ee}$, nous avons~$R^{\frac{1+\ee}{2+\ee} + O(1/\log\log R))}\ll_\ee \Psi(R, y)$, et ainsi
$$ S_1 + S_2 \ll_\ee \e^{O_\kappa(u)}\frac{\Psi(R, y)}{R} \ll \e^{O_\kappa(u)} \rhod(u)  $$
comme annoncé, en utilisant encore l'estimation~\eqref{eq:compar-psi-rho-unif}.

\end{proof}

\section{Cas d'un facteur linéaire}

Nous étudions dans cette section le cas~$d_1=1$ et~$r\geq 2$.

\subsection{Cas~$d_2=1$}

Le cas de deux facteurs linéaires présente une structure particulière, puisque dans ce cas nous pouvons utiliser des résultats d'équirépartition spécifiques.

\begin{theoreme}\label{th:correlation-2lin}
Soit~$\ee>0$ fixé. Il existe~$C, \delta>0$ tels que pour des entiers~$a$, $b$, $c$, $d$ de valeurs absolues au plus~$x^{\delta}$, avec~$a>0$, $c>0$, $\Delta := ad - bc\neq 0$ et $(a, c)=1$, l'on ait
\begin{equation}\label{eq:card}
\begin{aligned}
\card\{n\in[x, (1+\eta)x]: &\ P^+(an+b) \leq y_1,\ P^+(cn+d)\leq y_2 \} \\ &\ll_\ee \eta \Psi(x, y_1) \rhod(u_2)^{3/5-\ee}
\end{aligned}
\end{equation}
pour~$(\log x)^C\leq y_1 \leq y_2 \leq x$, $\e^{-\delta\sqrt{\log y_1}} \leq \eta \leq 1$ et~$y_2 \leq y_1^C$. Ici,~$u_j = (\log x)/\log y_j$.
\end{theoreme}

De cette estimation, combinée avec le théorème~2.(ii) et le corollaire~2 de Hildebrand et Tenenbaum~\cite{HT}, nous déduisons le corollaire suivant.

\begin{corollaire}
Pour tout~$\ee>0$, il existe~$C, \delta >0$ pouvant dépendre de~$\ee$ tels que lorsque~$(\log x)^C \leq y \leq x$, $\e^{-\delta\sqrt{\log y}}\leq \eta \leq 1$ et~$1\leq a, |b| \leq x^\delta$, l'on ait
$$ |\{n\in[x, (1+\eta)x]:\ P^+(n(an+b)) \leq y\}| \ll_\ee \eta x \rhod(u)^{8/5-\ee}. $$
\end{corollaire}

\begin{proof}[Démonstration du Théorème~\ref{th:correlation-2lin}]
Nous pouvons supposer que~$(a, b)=(c, d)=1$ sans perte de généralité. Supposons de plus~$y_2\leq x^{1/C}$ puisque dans le cas contraire la borne énoncée est triviale. Notons~$S$ le membre de gauche de~\eqref{eq:card}. Nous écrivons
$$ S = S_1 + S_2, $$
où~$S_1$ est la contribution des entiers~$n$ satisfaisant~$\prod_{p^\nu \| cn+d, \ p\nmid ac\Delta}p^\nu \geq x^{3/5-\ee}$, et~$S_2$ la contribution complémentaire.

Nous majorons d'abord~$S_1$. La Proposition~\ref{prop:poids} avec~$D=ac|\Delta|$,~$R=x^{3/5-\ee}$ et~$\kappa=2$, fournit l'existence d'une fonction~$\vartheta$ satisfaisant les conditions~(i)-(iv) de la Proposition~\ref{prop:poids}, en particulier
$$ S_1 \leq \sum_q \vartheta(q) \ssum{x \leq n \leq (1+\eta)x \\ P^+(an+b)\leq y_1 \\ q|cn+d} 1. $$
La relation de congruence s'écrit~$c(an+b) \equiv -\Delta \mod{aq}$. Par construction de~$\vartheta$, nous avons~$(q, ac\Delta)=1$. Notons~$X_1 := ax+b$, $X_2=(1+\eta)ax+b$ et~$\cI = [X_1, X_2]$ ; on remarque que~$(X_2-X_1)/X_1 \asymp \eta$. Alors
\begin{equation}
S_1 \leq \sum_{(q, ac\Delta)=1} \vartheta(q) \Psi(\cI, y_1 ; aq, -\bar{c}\Delta) \label{eq:majoS1}
\end{equation}
Nous appliquons l'estimation~\eqref{eq:equidistrib-35} du Théorème~\ref{th:equidistrib}. Considérons d'abord le terme principal~$T_1$ qui en résulte. On obtient
$$ T_1 = \sum_{(q, ac\Delta)=1} \vartheta(q)\frac{\Psi_{aq}(\cI, y_1)}{\vphi(aq)} = \ssum{m\in \cI \\ P^+(m)\leq y_1 \\ (m, a)=1}  \sum_{(q, mac\Delta)=1} \frac{\vartheta(q)}{\vphi(aq)}. $$
L'estimation (iv) de la Proposition~\ref{prop:poids} permet de majorer la somme sur~$q$ par
$$ \frac{\e^{O(u_2)}}{\vphi(a)} \frac{\Psi(R, y_2)}{R}. $$
Nous avons donc
$$ T_1 \ll \frac{\e^{O(u_2)}}{\vphi(a)} \frac{\Psi(R, y_2)}{R} \Psi_a(2ax+b, y_1). $$
Le théorème~2.4 de~\cite{BT05} et le théorème~4 de~\cite[theorem 4]{Hildebrand-Short} permettent d'écrire
$$ \Psi_a(2ax+b, y_1) \ll \prod_{p|a,\ p\leq y_1}(1-p^{-\alpha(x, y_1)}) \Psi(2ax+b, y_1) \ll a\prod_{p|a,\ p\leq y_1}(1-p^{-1}) \Psi(x, y_1). $$
Le théorème de Mertens et les inégalités~$y_1 \geq \log x$ et~$\omega(a)\ll \log x$ fournissent aisément~$\prod_{p|a,\ p\leq y_1}(1-p^{-1}) \ll \vphi(a)/a$. Nous avons donc obtenu
$$ T_1 \ll \e^{O(u_2)} \frac{\Psi(R, y_2)}{R} \Psi(x, y_1). $$
Puisque
$$ \e^{O(u_2)}\frac{\Psi(R, y_2)}{R} \ll_\ee \rhod(u_2)^{3/5-2\ee}, $$
nous avons bien la majoration requise, en utilisant~\eqref{eq:compar-psi-rho-unif}.

Nous écrivons ensuite
$$ U_1 := S_1-T_1 \leq \sum_{(q, ac\Delta)=1} |\vartheta(q)| \Big|\Psi(\cI, y_1 ; aq, -\bar{c}\Delta) - \frac{\Psi_{aq}(\cI, y_1)}{\vphi(aq)}\Big|. $$
La Proposition~\ref{prop:poids}.(i) fournit
$$ U_1 \leq \ssum{R \leq q \leq y^4 R \\ P^+(q)\leq y_2 \\ (q, ac\Delta)=1} \tau_3(q) \Big| \Psi(\cI, y_1 ; aq, -\bar{c}\Delta) - \frac{\Psi_{aq}(\cI, y_1)}{\vphi(aq)} \Big|. $$
L'estimation~\eqref{eq:equidistrib-35} du Théorème~\ref{th:equidistrib} montre que la somme sur~$q$ est
$$ \ll_{\ee, A} \eta \frac{\Psi(ax+b, y_1)}{a} \frac{\Psi(R, y_2)}{R}\frac{\e^{O(u_2)}}{(\log x)^A}. $$
pour tout~$A>0$. On en déduit, de même que précédemment et lorsque~$C$ est choisi suffisamment grand,
\begin{equation}
U \ll_\ee \eta \Psi(x, y_1) \rhod(u_2)^{3/5-\ee} .\label{eq:majo-final-U}
\end{equation}
Ce majorant est de nouveau de l'ordre de grandeur de la borne annoncée.

Enfin, concernant la contribution~$S_2$, nous avons trivialement
\begin{equation}
S_2 \leq \ssum{n\leq 2x \\ P^+(cn+d)\leq y \\ (cn+d, (ac|\Delta|)^\infty) \geq x^{1/5}} 1 \leq \ssum{n\leq 2x^{1+\delta} \\ (n, (ac|\Delta|)^\infty)\geq x^{1/5}} 1 \ll x^{1-\delta}\label{eq:majo-DD-linlin}
\end{equation}
si~$\delta$ est suffisamment petit. Cela est bien de l'ordre de grandeur annoncé en vertu de nos hypothèses sur~$(\eta, y_1, y_2)$.
\end{proof}

\subsection{Cas~$d_2>1$}

\begin{theoreme}\label{th:linQ}
Soient~$\ee>0$,~$(a, b)\in\NN^2$ et~$Q\in\ZZ[X]$ fixés. Il existe~$C, \delta>0$ tels que lorsque~$a\neq 0$ et~$Q$ est de degré~$g\geq 1$, l'on ait
\begin{equation}\label{eq:majo-linQ}
\begin{aligned}
|\{n\leq x :\ {}& P^+((an+b)Q(n))\leq y \}| \ll_{\ee, Q, a, b} x\rhod(u)^{3/2-\ee}
\end{aligned}
\end{equation}
pour~$(\log x)^C \leq y \leq x$.
\end{theoreme}

\begin{proof}
Nous supposons sans perte de généralité que~$Q$ irréductible, non proportionnel à~$aX+b$, que $(a, b)=1$ et que~$y\leq x^{1/10}$. Soit~$D^*$ le discriminant de~$Q$, ainsi que
\begin{equation}
\label{eq:def-D}
D := a^{g+1} Q(-b/a) D^*
\end{equation}
De même que dans la preuve du Théorème~\ref{th:correlation-2lin}, nous écrivons le membre de gauche de~\eqref{eq:majo-linQ} sous la forme~$S_1+S_2$, avec~$S_1$ la contribution des~$n$ tels que~$Q(n)/(Q(n), D^{\infty})\geq x^{1/2-\ee}$, et~$S_2$ la contribution complémentaire.

Nous étudions d'abord~$S_1$. Nous utilisons la Proposition~\ref{prop:poids} avec~$R=x^{1/2-\ee}$, l'entier~$D$ défini en~\eqref{eq:def-D} et~$\kappa = g$. Nous obtenons alors
\begin{align*}
S_1 {}&\leq \sum_q \vartheta(q) \ssum{n \leq x \\ P^+(an+b)\leq y \\ Q(n) \equiv 0\mod{q}}1 \\
{}&= O_\ee(x^{\ee}|b|) + \sum_q \vartheta(q) \ssum{\lambda\mod{q} \\ Q(\lambda)\equiv 0\mod{q}} \Psi(ax+b, y ; aq, a\lambda+b).
\end{align*}
Notons que nécessairement~$(a, a\lambda+b)=1$. De plus, on a~$(q, a\lambda+b) | a^gQ(-b/a)$. Par définition de~$D$ et la Proposition~\ref{prop:poids}.(ii), cela fournit~$(q, a\lambda+b)=1$; nous sommes donc en mesure d'appliquer l'estimation~\eqref{eq:equidistrib-12}. Notons
$$ \rhoQ_Q(q) := |\{\lambda\mod{q}: Q(\lambda)\equiv 0\mod{q}\}|. $$
Nous écrivons
$$ S_1 \ll_\ee x^\ee |b| + T_1 + U_1, $$
$$ T_1 = \sum_q \vartheta(q) \rhoQ_Q(q) \frac{\Psi_{aq}(ax+b, y)}{\vphi(aq)}, $$
$$ U_1 = \ssum{R \leq q \leq y^4R \\ P^+(q) \leq y \\ (q, D)=1} \tau_3(q) \rhoQ_Q(q) \max_{(c, q)=1} \Big| \Psi(ax+b, y ; aq, c) - \frac{\Psi_q(ax+b, y)}{\vphi(aq)}\Big|. $$
Notons que~$\rhoQ_Q(q) \leq g^{\omega(q)} $ lorsque~$(q, D^*)=1$ (voir par exemple la formule~(2.5) de~\cite{Henriot}). Nous avons d'une part, grâce à la Proposition~\ref{prop:poids}.(iv),
\begin{align*}
T_1 {}& = \frac1{\vphi(a)} \ssum{m\in S(ax+b, y) \\ (m, a)=1} \sum_{(q, m)=1} \frac{\vartheta(q) \rhoQ_Q(q)}{\vphi(q)}
\\ {}& \ll_{Q, a} \e^{O(u)} \frac{\Psi(ax+b, y)\Psi(R, y)}{R} 
\\ {}& \ll_{\ee, Q, a, b} x \rhod(u)^{3/2-2\ee}
\end{align*}
qui est de l'ordre de grandeur annoncé, quitte à diminuer~$\ee>0$. D'autre part, l'estimation~\eqref{eq:equidistrib-12} fournit pour tout~$A>0$ fixé
$$ U_1 \ll_{\ee, a} \frac{\e^{O(u)}\Psi(ax+b, y)\Psi(R, y)}{(\log x)^A R} $$
ce qui est à nouveau l'ordre de grandeur annoncé.

En ce qui concerne~$S_2$, qui est pour rappel la contribution au membre de gauche de~\eqref{eq:majo-linQ} des entiers~$n$ avec~$Q(n)/(Q(n), D^\infty)\leq R$, nous notons tout d'abord que pour~$x$ suffisamment grand, nous avons~$Q(n)\geq x^{g-1/2}$ pour~$n\geq x^{1-1/(3g)}$. De plus, on a pour une certaine constante~$C = C_Q$ la majoration~$\rhoQ_Q(m) \leq C^{\omega(m)}$ (voir par exemple la formule (2.14) de~\cite{Tenenbaum-ES}). Ainsi, de façon similaire à~\eqref{eq:majo-DD-linlin}, nous avons
$$ S_2 \leq \ssum{m|D^\infty \\ m\geq x^{g-1}} \ssum{n\leq x \\ m| Q(n)} 1 \ll_Q \ssum{m|D^\infty \\ x \leq m\ll x^g} C^{\omega(m)}\Big(\frac xm + 1\Big) \ll_Q x^{1/2}, $$
par la majoration de Rankin et l'inégalité~$\sum_{m|D^\infty} C^{\omega(m)} m^{-1/3g} \ll_Q 1$.
\end{proof}

\section{Cas d'un facteur quadratique,~$d_1=2$}

Dans le cas quadratique, nous montrons la majoration suivante, dont l'exposant dépasse le seuil~$c_Q=1$.
\begin{theoreme}\label{th:majo-quad}
Soit~$\delta$ réel donné, $0<\delta < 25/178$, et~$Q\in\ZZ[X]$ quadratique irréductible. Lorsque~$C>0$ est suffisamment grande en fonction de~$\delta$ et~$Q$, nous avons
\begin{equation}\label{eq:majo-quad}
\begin{aligned}
|\{n\leq x :\ {}& P^+(Q(n))\leq y \}| \ll_{\delta, Q} x\rhod(u)^{1+\delta}. 
\end{aligned}
\end{equation}
pour~$(\log x)^C \leq y \leq x$.
\end{theoreme}

Notre démonstration de ce résultat utilise l'équirépartition des racines des congruences quadratiques, due notamment à Hooley~\cite{HooleyQuad}. Notons
\begin{equation}
\theta \in [0, 7/64]\label{eq:def-RStheta}
\end{equation}
une borne en direction de la conjecture de Selberg-Ramanujan~\cite[chapitre~15.5]{IK}. La majoration~$\theta\leq 7/64$ est due à Kim et Sarnak~\cite{Kim}. Nous montrons le résultat suivant.

\begin{theoreme}\label{th:exporep-quad}
Soient~$\ee>0$,~$x, M, N\geq 1$, $MN\leq x^2$, et~$(a_m), (b_n)$ deux suites complexes bornées en module par~$1$. Soient~$D\in\ZZ$ qui n'est pas un carré d'entier, et~$V:\RR\to\CC$ une fonction lisse à support compact inclus dans~$\RR_+^\ast$. Alors
\begin{equation}
\begin{aligned}
\sum_{M\leq m \leq 2M} \ssum{N\leq n \leq 2N \\ (n, m)=1} a_m b_n {}& \Big(\ssum{k\in\NN \\ mn | k^2 - D} V\Big(\frac kx\Big) - x{\hat V}(0) \frac{\rhoQ(mn)}{mn}\Big) \\ & \ll_{\ee, D, V} x^{\ud+\ee}M^{\ud} + x^{1+\ee} N^{\frac32-\theta} M^{-\frac14+\frac\theta2},
\end{aligned} \label{eq:exporep-quad}
\end{equation}
où nous avons noté~${\hat V}(\xi) = \int_\RR V(t)\e(-t\xi)\dd t$ pour~$\xi\in\RR$.
\end{theoreme}
Notons que nous avons en toute circonstance la majoration triviale~$O(x^\ee MN)$. En utilisant~$\theta\leq 7/64$, et avec le choix~$M\leq x^{1-\eta/2}$, $N\leq x^{25/178-\eta/2}$, nous obtenons le Théorème~\ref{th:exporep-quad-intro} cité dans l'introduction. Le premier résultat de ce type est dû à Iwaniec, qui considère dans~\cite{Iwaniec} le cas~$D=-4$ et montre un exposant de répartition~$1+1/15$. Une extension à un polynôme quadratique quelconque est montrée par Lemke-Oliver dans~\cite{RJLO}. Iwaniec mentionne que la conjecture~$(R^\ast)$ de Hooley, qui porte sur des majorations de sommes courtes d'exponentielles, permettrait d'obtenir l'exposant~$1/9$~; notons que~$25/178 > 1/9$. Notre amélioration est basée sur l'utilisation dans ce contexte de méthodes issues de la théorie spectrale des formes automorphes~\cite{DFIWeyl, Toth}. La conjecture de Selberg~$\theta=0$ entraînerait la validité de la majoration~\eqref{eq:majo-quad} pour tout~$\delta<1/6$.

Le Théorème~\ref{th:exporep-quad} peut être généralisé au cas d'un polynôme quadratique irréductible quelconque, cependant dans le présent travail nous tirons parti de la possibilité de nous réduire \textit{a priori} au cas de~$X^2-D$, ce qui clarifie la présentation.

La preuve du Théorème~\ref{th:exporep-quad} étant indépendante du reste du présent travail, nous reportons sa démonstration à la section~\ref{sec:equidist}.

\begin{proof}[Démonstration du Théorème~\ref{th:majo-quad} à partir du Théorème~\ref{th:exporep-quad}]

Soient~$\eta, \delta>0$ des paramètres avec $\delta \leq 1/6$. Nous posons~$Q(X) = aX^2 + bX + c$, et notons~$S$ la somme du membre de gauche de~\eqref{eq:majo-quad}. Nous supposons sans perte de généralité que~$4a\leq y\leq x^{\eta}$, la borne annoncée étant triviale dans le cas contraire. De plus, par un découpage dyadique, nous pouvons supposer que les entiers~$n$ comptés dans~$S$ satisfont~$x/2 < n \leq x$. Enfin, la relation~$4a(aX^2 + bX + c) = (2aX+b)^2+4ac-b^2$ nous permet de supposer sans perte de généralité que~$Q(X) = X^2-D$, où~$D=b^2-4ac$ n'est pas un carré d'entier. Nous posons~$D^* = 2D$.

De même que dans la preuve du Théorème~\ref{th:linQ}, nous écrivons le membre de gauche de~\eqref{eq:majo-quad} comme~$S_1+S_2$, où~$S_1$ est la contribution des entiers~$n$ tels que
\begin{equation}
\prod_{\substack{p^\nu \| n^2-D \\ p\nmid D^*,  \nu\leq 2}} p^\nu \geq x^{6/5},\label{eq:cond-quad-S1}
\end{equation}
et~$S_2$ est la contribution complémentaire. Lorsque l'entier~$n$ est compté dans~$S_2$, alors l'une des deux inégalités suivantes est satisfaite~:
$$ \prod_{\substack{p^\nu \| n^2-D \\ p|D^*}} p^\nu \gg_D x^{1/5}, \qquad \text{ou} \qquad \prod_{\substack{p^\nu \| n^2-D \\ \nu\geq 3}} p^\nu \geq x^{1/5}. $$
Le nombre des entiers satisfaisant la première inégalité est majoré par~$O_D(x^{9/10})$ de façon similaire à la preuve du Théorème~\ref{th:linQ}. Le nombre des entiers satisfaisant la seconde inégalité est majoré, en suivant le raisonnement de~\cite[lemme~3.7]{Tenenbaum-ES}, par
\begin{equation}
\ssum{x^{1/5}\leq m \ll x^2 \\ p^\nu\| m \Rightarrow \nu\geq 3} \ssum{n\leq x \\ m|n^2-D} 1 \ll_D \sum_m C^{\omega(m)}\Big(\frac xm + 1\Big) \ll_{\ee, D} x^{13/15+\ee} + x^{2/3+\ee}.\label{eq:quad-majo-S2}
\end{equation}
Cela fournit un terme d'erreur acceptable.

Nous considérons maintenant~$S_1$. Nous fixons une fonction~$V$ de classe~$\cC^\infty$ sur~$\RR_+$ satisfaisant
$$ \1_{[1/2, 1]} \leq V \leq \1_{[1/3, 2]}. $$
Nous avons ainsi
$$ S_1 \leq \ssum{P^+(n^2-D)\leq y \\ n\text{ vérifie \eqref{eq:cond-quad-S1}}} V\Big(\frac nx\Big). $$
Nous appliquons la Proposition~\ref{prop:poids} avec les paramètres~$R=x^{1+\delta}$ (nous rappelons que~$\delta\in[0, 1/6]$), $\kappa=2$ pour la fonction~$f(q) = \rhoQ_Q(q)$, et avec~$n$ remplacé par~$\prod_{p^\nu \| n^2-D, p\nmid D, \nu\leq 2} p^\nu$. Cette dernière quantité est bien supérieure à~$R$ grâce à l'hypothèse~\eqref{eq:cond-quad-S1}. Nous obtenons donc l'existence d'une fonction~$\vartheta$ satisfaisant les propriétés~(i)-(iv) de la Proposition~\ref{prop:poids}. Puisque~$\1_{P^+(n)\leq y} \leq \1_{P^+(\prod_{p^\nu \| n^2-D, p\nmid D, \nu\leq 2} p^\nu)\leq y}$, nous en déduisons
\begin{align*}
S_1 \leq{}& \ssum{n\leq x \\ n\text{ vérifie \eqref{eq:cond-quad-S1}}} V\Big(\frac nx\Big) \sumb{q|n^2-D} \vartheta(q) \\
\leq {}& O_D(x^{14/15}) + \sumb{q} \vartheta(q) \ssum{n \in\NN \\ q|n^2-D} V\Big(\frac nx\Big) \numberthis\label{eq:ineg-quad-S11},
\end{align*}
où le symbole~$\sum^\flat$ indique que la somme est restreinte aux entiers~$q$ tels que~$p^\nu\|q\Rightarrow \nu\leq 2$. Dans la seconde inégalité, nous avons réintégré les entiers~$n$ qui ne satisfont pas~\eqref{eq:cond-quad-S1}, à l'aide de la borne~\eqref{eq:quad-majo-S2}. Soit~$S'_1$ la somme du membre de droite de~\eqref{eq:ineg-quad-S11}. Le terme principal attendu est
$$ T'_1 = x\hV(0) \sumb{q} \frac{\vartheta(q)\rhoQ(q)}q \ll \e^{O(u)} x \frac{\Psi(x^{1+\delta}, y)}{x^{1+\delta}}. $$

Nous étudions ensuite la différence
\begin{equation}
S'_1 - T_1 = \sumb{q} \vartheta(q) \Big(\ssum{n\in\NN \\ q|n^2-D} V\Big(\frac nx\Big) - x\hV(0)\frac{\rhoQ(q)}{q}\Big).\label{eq:diff-S1-majoquad}
\end{equation}
Nous rappelons que~$\vartheta$ est définie en~\eqref{eq:def-theta}, et nous considérons la variable de sommation~$d_0$ qui y intervient. Posons~$Q_1, Q_2\geq 1$ deux réels avec~$Q_1Q_2 = R$. Puisque~$d_0\geq R$, nous pouvons factoriser de façon unique~$d_0 = m\ell$, avec
$$ \begin{cases} Q_2 \leq \ell < Q_2 P^-(\ell)^\nu \quad (P^-(\ell)^\nu \| \ell), \\ P^+(m)<P^-(\ell). \end{cases} $$
Notons que par notre restriction sur la somme~$\sum_q^\flat$, nous avons~$P^-(\ell)^\nu \leq y^2$. Nous avons donc
$$ \vartheta(q) = \ssum{q = m\ell d_1 \\ R \leq m\ell < R P^+(\ell) \\ Q_2 \leq \ell < Q_2 P^+(\ell)^\nu \\ P^+(m) < P^-(\ell),\ P^+(\ell)\leq y \\ (m\ell d_1, D^*)=(m\ell, d_1)=1} \xi_{d_1} $$
où l'entier~$\nu$ dans la somme est déterminé par~$P^+(\ell)^\nu \|\ell$. Nous remarquons que dans cette somme, le poids de crible $\xi_{d_1}$ dépend implicitement de~$P^+(m\ell)=P^+(\ell)$~: ainsi, il ne dépend pas de~$m$. Nous séparons les variables~$m$ et~$\ell$ au moyen, par exemple, du lemme~13.11 de~\cite{IK} (avec~$z=x$), puis nous renommons~$q_1\gets m$ et~$q_2\gets \ell d_1$. Nous obtenons donc l'existence d'une fonction~$W:\RR^3\to\CC$ telle que pour toute fonction~$F:\RR\to\CC$,
$$ \int_{\RR^3} |W({\bf t})|\dd{\bf t} \ll (\log x)^3, \qquad \sumb{q} \vartheta(q) F(q) = \int_{\RR^3} W({\bf t}) T_F({\bf t}) \dd {\bf t}, $$
où nous avons posé~${\bf t}=(t_1, t_2, t_3)$ et
\begin{align*}
T_F(\mathbf{t}) {}&= \ssumm{y^{-2}Q_1 \leq q_1 \leq yQ_1 \\ Q_2 \leq q_2 \leq y^2 Q_2 \\ (q_1, q_2)=1} \alpha_{q_1, {\bf t}} \beta_{q_2, {\bf t}} F(q_1q_2), \\
\alpha_{q_1,\mathbf{t}} {}&= \1_{\substack{(q_1, D^*)=1\\p^\nu\|q_1\Rightarrow \nu\leq2}} q_1^{i(t_1-t_2)} P^+(q_1)^{-it_3}, \\
\beta_{q_2,\mathbf{t}} {}&= \1_{\substack{(q_2, D^*)=1\\p^\nu\|q_2\Rightarrow \nu\leq2}} \ssum{q_2=\ell d_1 \\ (\ell, d_1) = 1 \\ Q_1 < \ell \leq Q_1P^+(\ell)^\nu \\ P^+(\ell)\leq y} \xi_{d_1} \ell^{i(t_1-t_2)}P^+(\ell)^{it_2} P^-(\ell)^{it_3}.
\end{align*}
Nous appliquons cela avec~$F(q)$ remplacé par la parenthèse intérieure du membre de droite de~\eqref{eq:diff-S1-majoquad}. Uniformément pour~${\bf t}\in\RR^3$, et~$(Q_1', Q_2')\in\RR^2$ satisfaisant~$y^{-2}Q_1 /2 \leq Q_1' \leq yQ_1$ et~$Q_2/2\leq Q_2' \leq y^2Q_2$, le Théorème~\ref{th:exporep-quad} ainsi que la borne~$\tau(q_2)\ll Q_2^\ee$ fournissent
$$ \ssumm{Q'_1 < q_1 \leq 2Q'_1 \\ Q'_2 < q_2 \leq 2Q'_2 \\ (q_1, q_2)=1} \alpha_{q_1, \mathbf{t}} \beta_{q_2, \mathbf{t}} \Big(\ssum{n\in\NN \\ q_1q_2|n^2-D} V\Big(\frac nx\Big) - x\hV(0)\frac{\rhoQ(q_1q_2)}{q_1q_2}\Big)
\ll_{\ee, D} y^{O(1)}x^{1+\ee}(x^{-\ud}Q_1^{\ud} + Q_1^{-\frac14+\theta/2}Q_2^{\frac32-\theta}). $$
Nous optimisons par le choix~$Q_1 = x^{1-K\eta}$ et~$Q_2=x^{\kappa-K\eta}$ avec~$\kappa=\tfrac{1-2\theta}{6-4\theta}$ et~$K>0$ une constante absolue suffisamment grande. Cela fixe donc
\begin{equation}\label{eq:expr-delta}
\delta = 1+ \frac{1-2\theta}{6-4\theta} -2K\eta.
\end{equation}
Notre hypothèse~$y\leq x^\eta$, une somme dyadique sur~$Q_1'$ et~$Q_2'$, et une intégration par rapport à~$(t_1, t_2, t_3)$ fournissent donc
$$ S'_1 - T'_1 \ll_D x^{1+O(\eta) - \frac12K\eta} \ll_D x^{1-\eta} $$
si~$K$ est suffisamment grande. Ce terme d'erreur est acceptable pour~$y\geq (\log x)^C$, avec~$C=C(\eta)$. Nous concluons en utilisant la borne de Kim-Sarnak~$\theta\leq 7/64$~\cite{Kim} dans~\eqref{eq:expr-delta}, et en choisissant~$\eta$ suffisamment petit.

\end{proof}

\section{Le cas~$d_1\geq 3$}

Dans cette section, nous utilisons ce qui précède pour retrouver des résultats de Khmyrova~\cite{Khmyrova} et Timofeev~\cite{Timofeev}, à l'uniformité en~$Q$ près.

\begin{theoreme}\label{th:Qgen}
Soit~$Q\in\ZZ[X]$ irréductible de degré~$g\geq 3$. Il existe alors~$C>0$ tel que l'on ait
\begin{equation}\label{eq:majo-Qgen}
\begin{aligned}
|\{n\leq x :\ {}& P^+(Q(n))\leq y \}| \leq \e^{O_Q(u)} x\rhod(u)
\end{aligned}
\avoidbreak
\end{equation}
pour~$(\log x)^C \leq y \leq x$. La constante implicite et~$C$ dépendent au plus de~$Q$.
\end{theoreme}

\begin{proof}
Soit~$D$ le discriminant de~$Q$, et~$S$ le membre de gauche de~\eqref{eq:majo-Qgen}. De façon similaire à précédemment, nous nous restreignons à compter la contribution~$S_1$ à~$S$ constituée des entiers~$n$ tels que~$Q(n)/(Q(n), D^\infty)\geq x$. Nous appliquons alors la Proposition~\ref{prop:poids} avec~$\kappa = \operatorname{deg}(Q)$ et~$R = x/y^5$. Nous avons alors
\begin{align*}
S_1 {}&
\leq \sum_q \vartheta(q) \ssum{\lambda\mod{q} \\ Q(\lambda)\equiv 0\mod{q}} \ssum{n\leq x \\ n\equiv \lambda\mod{q}} 1 \\ {}&
\ll x\Big|\sum_q \frac{\vartheta(q)\rhoQ_Q(q)}q\Big| + \ssum{x/y^5 \leq q \leq x/y \\ P^+(q)\leq y} \tau_3(q)\rhoQ_Q(q) \\ {}&
\leq \e^{O_Q(u)}x\rhod(u)
\end{align*}
comme annoncé.
\end{proof}

\section{Application aux entiers~$n$ tels que~$P^+(n)^2|n$}

Dans cette section, nous prouvons le Corollaire~\ref{cor:dkdl}. Notons
$$ \cL := \e^{\sqrt{(\log x)\log\log x}}. $$

\begin{lemme}\label{lem:cL}
Soient~$\ee>0$,~$\ee\leq a<b \leq \ee^{-1}$ et~$\alpha, \beta>0$ donnés. Lorsque~$x$ tend vers l'infini,
$$ \sum_{\cL^a \leq p \leq \cL^b} \frac{\rhod(\alpha\tfrac{\log x}{\log p})}{p^{1+\beta}} \ll_{\ee, \alpha, \beta} \cL^{-\gamma + o(1)},  \qquad \text{avec}\quad \gamma := \inf_{t\in[a, b]}\Big(\frac{\alpha}{2t}+\beta t\Big). $$
\end{lemme}
\begin{proof}
En notant~$u_p = (\log x)/\log p$, nous avons
$$ \sum_{\cL^a \leq p \leq \cL^b} \frac{\rhod(\alpha\tfrac{\log x}{\log p})}{p^{1+\beta}} \ll \cL^{o(1)} \sum_{\cL^a \leq p \leq \cL^b} \frac{p^{-\beta}u_p^{-\alpha u_p}}{p} \ll \cL^{-\gamma+o(1)} \sum_{\cL^a \leq p \leq \cL^b} \frac1{p} \ll \cL^{-\gamma+o(1)}. $$
\end{proof}

\begin{proof}[Démonstration du Corollaire~\ref{cor:dkdl}]
Nous remarquons tout d'abord qu'en procédant à un découpage dyadique, il suffit de montrer la majoration~\eqref{eq:majo-DKDL} avec la contrainte supplémentaire
\begin{equation}
x/2 < n \leq x.\label{eq:n-dyad}
\end{equation}
Notons~$E(x)=|\{x/2 < n \leq x : P^+(n)^2|n, P^+(n+1)^2|n+1\}|$. Pour toute paire de nombres premiers distincts~$(p, q)$, nous notons~$E_{p, q}(x)$ la contribution à~$E(x)$ des entiers~$n$ tels que~$p=P^+(n)$ et~$q=P^+(n+1)$, de sorte que
\begin{equation}\label{eq:E-somme}
E(x) = 2\ssum{p, q \text{ premiers} \\ p< q} E_{p, q}(x).
\end{equation}
Nous avons la borne triviale
$$ E_{p, q}(x) \leq \min\{\Psi(x/p^2, p), \Psi(x/q^2, q)\}. $$
La majoration
$$ \sum_{p\leq \cL^{1/6} \text{ ou } p\geq \cL^{3}} \Psi(x/p^2, p) \ll x\cL^{-3}+\Psi(x, \cL^{1/6}) \ll x\cL^{-2\sqrt{2}} $$
permet donc de restreindre la somme~\eqref{eq:E-somme} à~$p, q \in [\cL^{1/6}, \cL^3]$. Nous fixons dans ce qui suit une telle paire~$(p, q)$.

Le théorème de Bezout nous assure l'existence et l'unicité d'une paire~$(r, s)$ d'entiers tels que~$q^2r-p^2s=1$, avec~$1\leq r \leq p^2$ et~$1\leq s \leq q^2$. Les conditions~$p^2|n$ et~$q^2|n+1$ nous permettent de paramétrer~$n$ sous la forme
$$ n=p^2(s+q^2m), \qquad n+1 = q^2(r+p^2m). $$
Au vu de la condition~\eqref{eq:n-dyad}, nous avons~$x/(3(pq)^2) \leq m \leq x/(pq)^2$. Ainsi,
$$ E_{p, q}(x) \leq \big|\big\{m \in [\tfrac{x}{3(pq)^2}, \tfrac x{(pq)^2}] :\ P^+(q^2m+s)\leq p, P^+(p^2m+r)\leq q\big\}\big|. $$
Le Théorème~\ref{th:correlation-2lin} (avec~$\Delta=1$) fournit donc
$$ E_{p, q}(x) \ll_\ee \cL^{o(1)} \frac{x}{(pq)^2} \frac{\Psi(x, p)}{x} \frac{\Psi(x^{3/5-\ee}, q)}{x^{3/5-\ee}}. $$

Nous sommons cette borne sur~$(p, q)$ avec~$p<q$, et notons~$\kappa=3/5-\ee$. Lorsque~$x$ tend vers l'infini, le Lemme~\ref{lem:cL} fournit d'une part
$$ \sum_{p<q \leq \cL^{3}} \frac{\Psi(x^{\kappa}, q)}{q^2 x^{\kappa}} \ll \sum_{p<q \leq \cL^{3}} \frac{\rhod(\kappa\tfrac{\log x}{\log q})}{q^2} \ll
\begin{dcases}
\cL^{-\sqrt{2\kappa}+o(1)}, & (p\leq\cL^{\sqrt{\kappa/2}}), \\
\cL^{o(1)}\frac{\rhod(\kappa\tfrac{\log x}{\log p})}{p}, & (p\geq \cL^{\sqrt{\kappa/2}}),
\end{dcases} $$
et d'autre part
\begin{align*}
&\sum_{\cL^{1/6} \leq p \leq \cL^{\sqrt{\kappa/2}}} \frac{\rhod(\tfrac{\log x}{\log p})}{p^2} \ll \cL^{-\sqrt{1/(2\kappa)}-\sqrt{\kappa/2}+o(1)}, \\
&\sum_{\cL^{\sqrt{\kappa/2}}<p \leq \cL^3} \frac{\rhod(\tfrac{\log x}{\log p}) \rhod(\kappa\tfrac{\log x}{\log p})}{p^3} \ll \cL^{-2\sqrt{1+\kappa}+o(1)}.
\end{align*}
Puisque~$2\sqrt{1+\kappa}\leq (\sqrt{2}+1/\sqrt{2})\sqrt{\kappa}+1/\sqrt{2\kappa}$, nous obtenons
$$ E(x) \ll x\cL^{-2\sqrt{2}} + \ssum{\cL^{1/6} \leq p < q \leq \cL^3} E_{p, q}(x) \ll_\ee \cL^{-2\sqrt{1+\kappa}+o(1)}. $$
Le résultat annoncé suit en faisant tendre~$\ee$ vers~$0$.

\end{proof}

\section{Niveau de répartition de~$\{n^2-D\}$}\label{sec:equidist}

Dans cette section, nous montrons le Théorème~\ref{th:exporep-quad}. Celui-ci découle immédiatement de la proposition suivante, grâce à l'inégalité de Cauchy--Schwarz.
\begin{proposition}\label{prop:equidistrib}
Soient~$\ee>0$, $x, M, N\geq 1$, $MN\leq x^2$,~$(b_n)\in\CC^\NN$ avec~$\|b\|_\infty \leq 1$, $D\in\ZZ$ qui n'est pas un carré d'entier, et~$V:\RR\to\CC$ une fonction lisse à support compact inclus dans~$\RR_+^\ast$. Alors
\begin{equation}
\begin{aligned}
\sum_{M< m \leq 2M} \Big|\ssum{N< n \leq 2N \\ (n, m)=1} b_n {}& \Big(\ssum{k\in\NN \\ mn | k^2 - D} V\Big(\frac kx\Big) - x{\hat V}(0) \frac{\rhoQ(mn)}{mn}\Big)\Big|^2 \\
& \ll_{\ee, V, D} \Big(1 + x\Big(\frac M{N^2}\Big)^{-\frac32+\theta}\Big)x^{1+\ee}.
\end{aligned}
\label{eq:equidistrib}
\end{equation}
\end{proposition}

\begin{remarque}
Le membre de gauche de~\eqref{eq:equidistrib} est majoré trivialement par~$M^{-1}x^{2+\ee}$, ce qui permet en particulier de supposer d'emblée que~$M\geq N^2$, et justifie l'intérêt d'avoir une valeur de~$\theta$ aussi petite que possible.
\end{remarque}

La proposition précédente sera déduite du lemme suivant, qui concerne l'équirépartition des racines de congruences quadratiques.

\begin{lemme}\label{lemme:type1}
Soient~$(q, r, d)\in\NN^3$ avec~$(q, 2Dr)=1$ et~$d|q$, $\lambda\mod{d}$ une classe inversible, et~$\omega\mod{d}$ une classe de résidus telle que~$\omega^2\equiv D\mod{d}$. Soient~$M\gg qd$, $f$ une fonction lisse à support compact inclus dans~$\RR_+^\ast$, satisfaisant
\begin{equation}
\|f^{(j)}\|_\infty \ll_j 1,\label{eq:majo-der-f}
\end{equation}
soient~$0\leq \alpha < \beta < 1$ et
\begin{equation}
P_f(M ; q, r, d, \lambda, \omega, \alpha, \beta) := \ssum{(m, \Omega)\in \cD \\ \alpha \leq \frac\Omega{mq} < \beta} f\Big(\frac{m}{M}\Big),\label{eq:def-Pf}
\end{equation}
où~$\cD$ est l'ensemble des paires~$(m, \Omega)$ telles que
$$ (m, qr)=1, \quad m\equiv \lambda\mod{d} $$
$$ \Omega^2\equiv D\mod{mq}, \quad \Omega \equiv \omega \mod{d}. $$
Alors pour tout~$\ee>0$, l'on a
\begin{equation}
P_f(M ; q, r, d, \lambda, \omega, \alpha, \beta) = A_f(M ; q, r, d, \alpha, \beta) + O_{\ee, D, f}\big((qrM)^\ee d^{\frac34}(qd)^{\ud-\theta}M^{\ud+\theta}\big).\label{eq:type1}
\end{equation}
Ici, le terme principal~$A_f$ est défini par
$$ A_f(M ; q, r, d, \alpha, \beta) = (\beta-\alpha)M{\hat f}(0) C_D \frac{A(qr)\rhoQ(q/(q, d^\infty))}{\vphi(d)}, $$
où~$A(qr) = \prod_{p|qr}(1+1/p)^{-1}$ et~$C_D>0$ est une constante qui ne dépend que de~$D$. La constante implicite ne dépend que de~$\ee$,~$D$, et des constantes implicites dans~\eqref{eq:majo-der-f}.
\end{lemme}

Ce résultat découle des majorations de sommes d'exponentielles suivantes.
\begin{lemme}\label{lemme:somme-expo}
Sous les hypothèses et notations du Lemme~\ref{lemme:type1}, les majorations suivantes ont lieu pour tout~$\ee>0$.
\begin{enumerate}
\item  Pour~$1\leq |h| \leq qd^\ud$,
\begin{equation}
\sum_{(m, \Omega)\in \cD} f\Big(\frac{m}{M}\Big)\e\Big(\frac{h\Omega}{mq}\Big) \ll_{\ee, D, f} |h| (qr)^\ee + (rM)^\ee d^{\frac34}(qd, h)^\theta (qd)^{\ud-\theta} M^{\ud+\theta}. \label{eq:majo-expo}
\end{equation}
\item Pour~$\ud\leq H \ll qM$,
\begin{equation}
\frac1H\sum_{H< |h| \leq 2H} \Big|\sum_{(m, \Omega)\in \cD} f\Big(\frac{m}{M}\Big)\e\Big(\frac{h\Omega}{mq}\Big)\Big| \ll_{\ee, D, f} H (qr)^\ee + (rM)^\ee d^{\frac34} (qd)^{\ud-\theta} M^{\ud+\theta}. \label{eq:somme-expo-moyh}
\end{equation}
\item Supposons~$d=1$,~$\ud\leq Q\ll M$,~$\ud \leq H \ll QM$, que~$t\in[0, 1]$, et soit~$\cI\subset[H, 2H]$ un intervalle, et~$(f_q)_{Q<q\leq 2Q}$ une suite de fonctions de fonctions satisfaisant~\eqref{eq:majo-der-f} et $f_q(v)\neq 0 \Rightarrow v\asymp 1$ uniformément en~$q$. Alors
\begin{equation}
\begin{aligned}
\frac1Q \ssum{Q<q\leq 2Q\\ (q, 2Dr)=1}\Big|\frac1H \sum_{h\in \cI} \e(th){}& \sum_{(m, \Omega)\in \cD} f_q\Big(\frac{m}{M}\Big)\e\Big(\frac{h\Omega}{mq}\Big)\Big| \\
{}& \ll_{\ee, D, f} H (Qr)^\ee + (rM)^\ee\Big\{ M^{\ud} + H^{-\ud} Q^{\ud-\theta} M^{\ud+\theta}\Big\}. \label{eq:somme-expo-moyqh}
\end{aligned}
\end{equation}
\end{enumerate}
Les constantes implicites ne dépendent que de~$\ee$, $D$, et des constantes implicites dans~\eqref{eq:majo-der-f}.
\end{lemme}
\begin{remarques}\hfill
\begin{itemize}
\item Dans le cas~$d=1$, en utilisant la borne de Selberg~$\theta\leq1/4$ (\textit{cf.}~\cite{DI}, theorem~4), nous retrouvons, à l'uniformité en~$h$ près, un résultat de Duke, Friedlander et Iwaniec~\cite[formule~(25)]{DFIWeyl} et T\'{o}th~\cite[formule~(15)]{Toth}.
\item Les majorations~\eqref{eq:majo-expo} et~\eqref{eq:somme-expo-moyh} sont également valables pour~$M\ll qd$, mais elles sont alors moins précises que la majoration triviale~$O_{\ee, D, f}(q^\ee M^{1+\ee})$.
\end{itemize}
\end{remarques}

\subsection{Démonstration du \texorpdfstring{Lemme~\ref{lemme:somme-expo}}{Lemme somme-expo}}

Nous nous concentrons dans un premier temps sur la majoration~\eqref{eq:majo-expo}. Nous supposons sans perte de généralité que~$h>0$, quitte à considérer le nombre complexe conjugué.

\subsubsection{Coprimalité}

Soit~$S'(M, q, \lambda)$ le membre de gauche de~\eqref{eq:majo-expo}. Une inversion de Möbius fournit
\begin{equation}
S'(M, q, \lambda) = \ssum{\ell|qr \\ (\ell, d)=1}\mu(\ell) S(M/\ell, q\ell, \lambda\bar{\ell})\label{eq:mobius}
\end{equation}
où~$\mu$ désigne la fonction de Möbius et
$$ S(M, q, \lambda) := \ssum{m\in\NN \\ m\equiv \lambda\mod{d}}\ssum{\Omega \in \NN \\ \alpha mq\leq \Omega < \beta mq \\ \Omega^2 \equiv D \mod{mq} \\ \Omega \equiv \omega \mod{d}} f\Big(\frac{m}{M}\Big)\e\Big(\frac{h\Omega}{mq}\Big). $$
Il nous suffira donc de montrer que~$S(M, q, \lambda)$ est majorée par le membre de droite de~\eqref{eq:majo-expo}.

\subsubsection{Correspondance de Gauss}

Soit
$$ \sQ_D = \{Q(X, Y) = AX^2+2BXY+CY^2,\ (A, B, C)\in\ZZ^3,\ B^2-AC = D\}. $$
Pour~$Q\in\sQ_D$, nous notons~$(A(Q), B(Q), C(Q))$ les coefficients dans l'écriture ci-dessus. Le groupe~$\Gamma = PSL_2(\ZZ)$ agit sur~$\sQ_D$ par
$$ \sigma Q (x, y) = Q((x, y)\sigma) \qquad (\sigma\in\Gamma) $$
où le produit dans le membre de droite est le produit matriciel. En particulier,
\begin{equation}\label{eq:action-Gamma}
\begin{aligned}
B(\sigma Q) {}&= \alpha\gamma A + (\alpha\delta+\beta\gamma)B + \beta\delta C, \\
C(\sigma Q) {}&= \gamma^2A + 2\gamma\delta B + \delta^2 C = Q(\gamma, \delta)
\end{aligned}
\end{equation}
si~$\sigma=\left(\begin{smallmatrix} \alpha & \beta \\ \gamma & \delta \end{smallmatrix}\right)$. En raisonnant de façon identique à~\cite[p. 427]{DFIWeyl} (voir aussi~\cite[section~6.1]{Kow}), nous obtenons
\begin{equation}\label{eq:gauss}
S(M, q, \lambda) = \sum_{Q\in \Gamma\backslash\sQ_D} \ssum{\sigma\in \Gamma_\infty \backslash\Gamma/\Gamma_Q \\ \cP(\sigma)} f\Big(\frac{C(\sigma Q)}{qM}\Big) \e\Big(\frac{h B(\sigma Q)}{C(\sigma Q)}\Big),
\end{equation}
où~$\Gamma_\infty = \{\left( \begin{smallmatrix} 1 & n \\ 0 & 1 \end{smallmatrix} \right), n\in \ZZ)\}$ et~$\cP(\sigma)$ désigne la propriété
$$ \cP(\sigma) \Leftrightarrow \left\{ \begin{aligned}& C(\sigma Q) \equiv \lambda q \mod{qd}, \\ &B(\sigma Q) \equiv \omega \mod{d} ,\end{aligned}\right. $$
et~$\Gamma_Q\subset \Gamma$ est le stabilisateur de~$Q$.

\subsubsection{Localisation des variables}

Soit~$\sigma=\left(\begin{smallmatrix}\ast&\ast\\\gamma&\delta\end{smallmatrix}\right)$ un élément générique en indice de la somme du membre de droite de~\eqref{eq:gauss}. Nous introduisons, suivant T\'o{th}~\cite[lemme~4.2]{Toth}, une fonction~$\Psi :\Gamma \to \RR$ qui permet d'encoder le quotient par~$\Gamma_Q$. Cette fonction satisfait~$\sum_{\tau\in\Gamma_Q} \Psi(\sigma \tau) = 1$ pour tout~$\sigma \in\Gamma$. Dans le cas~$D<0$, la fonction~$\Psi$ est constante, et dans le cas contraire~$\Psi(\sigma)$ est une fonction~$\cC^\infty$ du rapport~$\delta/\gamma$. Nous nous donnons aussi une fonction lisse~$w:\RR\to\RR_+$ satisfaisant
$$ \1_{|t|\leq \ud}\leq w(t) \leq \1_{|t|\leq 2}, \qquad w(t)+w(1/t)=1 $$
pour~$t\neq 0$. Nous insérons le poids~$w(\gamma/\delta)+w(\delta/\gamma)$ dans le membre de droite de~\eqref{eq:gauss}. Dans la contribution du terme~$w(\gamma/\delta)$, nous appliquons aux sommes sur~$\sigma$ et~$Q$ les involutions changeant~$Q(X, Y)$ en~$\tQ = Q(Y, X)$, et~$\sigma$ en~$\tsigma = \left( \begin{smallmatrix}-\beta & -\alpha \\ \delta & \gamma \end{smallmatrix} \right)$. Nous avons alors
$$ B(\tsigma\tQ) = -B(\sigma Q), \qquad C(\tsigma \tQ) = C(\sigma Q). $$
Nous obtenons
\begin{equation}
S(M, q, \lambda) = S(h, \Psi_1) + S(-h, \Psi_2),\label{eq:SMq-Shpsi}
\end{equation}
avec
\begin{equation}
(\Psi_1(\sigma), \Psi_2(\sigma)) = (w(t)\Psi(t), w(t)\Psi(1/t)) \qquad (\sigma = \left(\begin{smallmatrix}\ast&\ast\\\gamma&\delta\end{smallmatrix}\right) \in \Gamma_\infty\backslash\Gamma, \quad t=\gamma/\delta),\label{eq:def-psi12}
\end{equation}
\begin{equation}
F_{\Psi, Q}(\sigma) = \Psi(\sigma) f\Big(\frac{C(\sigma Q)}{qM}\Big),\label{eq:def-Fpsi}
\end{equation}
et
\begin{equation}
S(h, \Psi) = \sum_{Q\in \Gamma\backslash\sQ_D}\ssum{\sigma\in \Gamma_\infty \backslash\Gamma \\ \cP(\sigma)} F_{\Psi, Q}(\sigma) \e\Big(\frac{h B(\sigma Q)}{C(\sigma Q)}\Big). \label{eq:def-Shpsi}
\end{equation}
Nous fixons dorénavant~$\Psi\in\{\Psi_1, \Psi_2\}$, en notant que cette fonction (donc la fonction~$F_{\Psi, Q}$ associée) est nulle dès que~$|t|\geq 2$ (avec la notation~\eqref{eq:def-psi12}).

\subsubsection{Simplification de la phase}

Nous écrivons la définition~\eqref{eq:def-Shpsi} comme~$S(h, \Psi) = \sum_{Q\in\Gamma\backslash\sQ_D} S_Q(h, \Psi)$. La somme sur~$Q$ est finie et son nombre de termes dépend au plus de~$D$. Il nous suffira donc de borner~$S_Q(h, \Psi)$ séparément pour chaque~$Q$. Pour~$\sigma\in\Gamma$, nous définissons
$$ \phi_\sigma = \frac\alpha\gamma \in\RR/\ZZ \qquad (\sigma = \left(\begin{smallmatrix} \alpha & \ast \\ \gamma & \ast \end{smallmatrix}\right), \quad \gamma\neq 0). $$
L'identité, due à Hooley~\cite[formule~(27)]{HooleyQuad},
\begin{equation}
\e\Big(\frac{hB(\sigma Q)}{C(\sigma Q)}\Big) = \e(h\phi_\sigma) + O(h(qM)^{-1})\label{eq:approx-triv-exp}
\end{equation}
est alors établie de façon similaire au lemme 4.3 de Toth~\cite{Toth}. En remplaçant dans~$S_Q(h, \Psi)$, nous obtenons
\begin{equation}
S_Q(h, \Psi) = T_Q(h, F_{\Psi, Q}) + O(h),\label{eq:SQ-TQ}
\end{equation}
avec
$$ T_Q(h, F) = \ssum{\sigma\in \Gamma_\infty \backslash\Gamma \\ \cP(\sigma)} F(\sigma)\e(h\phi_\sigma). $$

\subsubsection{Conditions de congruence}

Nous décomposons par le sous-groupe de congruence de Hecke~$\Gamma_0(qd)$ pour obtenir
$$ T_Q(h, \Psi) = \ssumm{\tau \in \Gamma_\infty \backslash\Gamma_0(qd) \\ \sigma \in \Gamma_0(qd) \backslash \Gamma \\ \cP(\tau\sigma)} F(\tau\sigma)\e(h \phi_{\tau\sigma}). $$
Si~$\tau = \left(\begin{smallmatrix} \alpha & \beta \\ \gamma & \delta \end{smallmatrix} \right)\in\Gamma_0(qd)$, alors les relations~\eqref{eq:action-Gamma} ainsi que~$qd|\gamma$ montrent que
$$ \cP(\tau\sigma) \Leftrightarrow \left\{ \begin{aligned} & q \mid C(\sigma Q), \\ & \delta^2 q^{-1} C(\sigma Q) \equiv \lambda \mod{d}, \\ & B(\sigma Q) \equiv \omega \mod{d}. \end{aligned} \right. $$
Puisque~$(\lambda, d)=1$, la seconde condition est détectée par des caractères de Dirichlet, ce qui fournit
\begin{equation}
T_Q(h, \Psi) = \frac1{\vphi(d)}\sum_{\chi \mod{d}} \bar{\chi(\lambda)} \ssum{\sigma \in \Gamma_0(qd) \backslash \Gamma \\ \cP^*(\sigma)}\chi(q^{-1}C(\sigma Q)) U_{Q, \sigma}(h, \Psi),\label{eq:def-TQ}
\end{equation}
avec
$$ U(h, \Psi) = U_{Q, \sigma}(h, \Psi) = \sum_{\tau \in \Gamma_\infty \backslash\Gamma_0(qd)} \bar{\vartheta(\tau)} F(\tau\sigma) \e(h \phi_{\tau\sigma}), $$
où~$\cP^*(\sigma)$ désigne maintenant les conditions
\begin{equation}
\cP^*(\sigma) \Leftrightarrow \left\{\begin{aligned} & q\mid C(\sigma Q) \\ & B(\sigma Q) \equiv \omega \mod{d} \end{aligned} \right.\label{eq:cond-Pstar}
\end{equation}
et~$\vartheta$ dénote le caractère central défini par
$$ \vartheta(\tau) = \bar{\chi}^2(\delta) \qquad (\tau = \left(\begin{smallmatrix} \ast & \ast \\ \ast & \delta \end{smallmatrix} \right) \in \Gamma_0(qd)). $$

\subsubsection{Rappels concernant les sommes de Kloosterman généralisées}

Dans cette section, nous rappelons quelques faits sur les sommes de Kloosterman. Nous référons aux chapitres 2 et 4 de~\cite{Iwaniec-Topics} pour les définitions. Soient deux pointes~$\ca, \cb \in \PP^1(\RR)$ pour l'action de~$\Gamma_0(qd)$, de stabilisateurs~$\Gamma_\ca$ et~$\Gamma_{\cb}$ et matrices d'échelle~$\sigma_\ca, \sigma_\cb \in PSL_2(\RR)$, c'est-à-dire telles que
$$ \Gamma_\ca = \sigma_{\ca}\Gamma_{\infty}\sigma_{\ca}^{-1}, \qquad \Gamma_\cb = \sigma_{\cb}\Gamma_{\infty}\sigma_{\cb}^{-1}. $$
Une pointe est équivalente sous l'action de~$\Gamma_0(qd)$ à une unique pointe~$\ca'=u/v$, avec
$$ v\geq 1, \quad v|qd, \quad(u, v)=1, \quad 1\leq u \leq (v, qd/v). $$
Nous pouvons donc définir la \textit{largeur} de la pointe~$\ca$ comme le nombre
\begin{equation}
w_\ca = \frac{q}{(q, v^2)}.\label{eq:def-largeur-pointe}
\end{equation}

Nous associons à~$(\ca, \cb)$ l'ensemble de modules
$$ \cC(\ca, \cb) := \Big\{c\in\RR_+^* :\ \exists \alpha, \beta, \delta \in\RR, \begin{pmatrix} \alpha & \beta \\ \gamma & \delta \end{pmatrix} \in \sigma_\ca^{-1}\Gamma\sigma_\cb \Big\} .$$
Pour tout~$c\in\cC(\ca, \cb)$ et~$(m, n) \in \ZZ^2$, nous définissons la somme de Kloosterman
\begin{equation}
S_{\ca\cb}(m, n ; \gamma) = \ssum{\delta\in[0, \gamma\mc{[} \ : \\  \left(\begin{smallmatrix}\alpha & \ast \\ \gamma & \delta\end{smallmatrix}\right) \in \sigma_\ca^{-1}\Gamma_0(qd) \sigma_{\cb}} \bar{\vartheta}(\sigma_\ca \big(\begin{smallmatrix}\alpha & \ast \\ \gamma & \delta\end{smallmatrix}\big) \sigma_\cb^{-1})\e\Big(\frac{\alpha m+\delta n}{\gamma}\Big).\label{eq:def-Kloo}
\end{equation}
Nous renvoyons à la section~4.1.1 de~\cite{D-Kuz} pour plus de détails, notamment sur la dépendance de~$S_{\ca\cb}(m, n ; c)$ vis-à-vis des matrices d'échelle~$(\sigma_\ca, \sigma_\cb)$. Nous utiliserons dans ce travail les faits suivants.

\begin{lemme}\label{lemme:pointe}
Soit~$\sigma\in\Gamma_0(qd) \backslash \Gamma$ satisfaisant les conditions~$\cP^\ast(\sigma)$ définies en~\eqref{eq:cond-Pstar}.
\begin{enumerate}
\item Le nombre de telles classes~$\sigma$ est majoré par~$O(d\tau(q))$.
\item Supposons que la pointe~$\ca = \sigma\infty$ soit équivalente à~$u/v$, avec~$v|q$,~$1\leq u<v$ et~$(u, v)=1$. Alors~$v|Q(0, 1)^2$, en particulier~$v=O_Q(1)$, et
\begin{equation}
w_\ca = \frac{qd}{(qd, v^2)} \asymp_Q qd.\label{eq:taille-wca}
\end{equation}
\item L'ensemble de modules~$\cC(\infty, \ca)$ s'écrit
$$ \cC(\infty, \ca) = \big\{w_\ca^{\ud}v m,\ m\in\ZZ:\ (m, qd/v)=1 \big\}. $$
\item Lorsque~$\gamma=w_\ca^{\ud}vm \in \cC(\infty, \ca)$, la somme de Kloosterman~$S_{\infty\ca}(h, n ; \gamma)$ admet l'expression
$$ S_{\infty\ca}(h, n ; \gamma) = \ssum{\alpha\mod{vm} \\ \delta\mod{u[v, v']m} \\ \delta\equiv m \mod{uv'} \\ (\delta-m, um) = u \\ \alpha\delta \equiv u\mod{vm}} \bar{\vartheta}\big(\begin{smallmatrix}\ast & \ast \\ \ast & \delta\end{smallmatrix}\big)\e\Big(\frac{h\alpha}{vm} + \frac{n\delta}{u[v, v']m}\Big), $$
où l'on a noté~$v' = qd/v$. Ici, les matrices d'échelles choisies sont
$$ \sigma_\infty = {\rm Id}, \qquad \sigma_{\ca} = \begin{pmatrix} u\sqrt{w_\ca} & 0 \\ v\sqrt{w} & (u\sqrt{w})^{-1} \end{pmatrix}. $$
\item Nous avons la borne triviale
\begin{equation}
|S_{\infty\ca}(h, n ; \gamma)| \leq \frac{v}{(v, v')}(m, u) m \ll_Q m,\label{eq:majotriv-kloosterman}
\end{equation}
\item Lorsque~$n=0$, nous avons
\begin{equation}\label{eq:majogauss-kloosterman}
|S_{\infty\ca}(h, 0 ; \gamma)| \leq \tau(2m)^{O_{\ca, Q}(1)} (dh, m).
\end{equation}
\end{enumerate}
\end{lemme}

La démonstration de ce lemme, qui est indépendante du reste de la démonstration du Lemme~\ref{lemme:somme-expo}, est reportée à la section~\ref{sec:demo-pointes}.

\subsubsection{Complétion de sommes}

Dans la somme~$U(h, \Psi)$, nous changeons~$\tau$ en~$\tau\sigma^{-1}$, de sorte que
\begin{equation*}
U(h, \Psi) = \sum_{\tau \in \Gamma_\infty \backslash\Gamma_0(qd)\sigma} \vartheta(\tau\sigma^{-1}) F(\tau) \e(h\phi_\tau).
\end{equation*}
La pointe~$\ca = \sigma\infty$ est équivalente à~$u/v$ pour un certain~$v|qd$ et~$(u, v)=1$, et cette écriture est unique si l'on impose~$1\leq u \leq (v, qd/v)$. Nous posons temporairement
$$ \tau_\ca = \begin{pmatrix} w_\ca^{1/2} & 0 \\ 0 & w_\ca^{-1/2}\end{pmatrix}, \qquad \sigma_\ca = \sigma \tau_\ca $$
de sorte que le stabilisateur~$\Gamma_\ca\subset\Gamma_0(qd)$ de~$\ca$ vérifie~$\Gamma_\ca = \sigma_\ca \Gamma_\infty \sigma_\ca^{-1}$. Dans la somme du membre de droite de~\eqref{eq:reecr-Uhpsi}, nous remplaçons encore~$\tau$ par~$\tau\tau_\ca^{-1}$ en remarquant que cela laisse la quantité~$\phi_\tau$ inchangée. Nous obtenons
\begin{equation}
U(h, \Psi) = \sum_{\tau \in \Gamma_\infty \backslash\Gamma_0(qd)\sigma_\ca} \vartheta(\tau\sigma_\ca^{-1}) F(\tau\tau_\ca^{-1}) \e(h\phi_\tau).\label{eq:reecr-Uhpsi}
\end{equation}

À ce stade, nous remarquons que~$(v, qd/v)|q$. En particulier, la pointe~$\ca$ est singulière pour~$\vartheta$, ce qui signifie
$$ \vartheta(\tau) = 1 \qquad (\tau \in \Gamma_\ca). $$
Nous séparons la somme sur~$\tau$ dans le membre de droite de~\eqref{eq:reecr-Uhpsi} suivant les classes à droite modulo~$\Gamma_\infty$. Notons que pour tout~$\omega\in \Gamma_\infty$, nous avons~$\phi_{\tau\omega} \equiv \phi_\tau\mod{1}$, ainsi que
$$ \vartheta(\tau\omega\sigma_\ca^{-1}) = \vartheta(\tau\sigma_\ca^{-1})\vartheta(\sigma_\ca\omega\sigma_\ca^{-1}) = \vartheta(\tau\sigma^{-1}). $$
Nous obtenons
$$ U(h, \Psi) = \sum_{\tau \in \Gamma_\infty \backslash\Gamma_0(qd)\sigma_\ca / \Gamma_\infty} \vartheta(\tau\sigma_\ca^{-1}) \e(h\phi_\tau) \sum_{k\in\ZZ} F\left(\tau\left(\begin{smallmatrix} 1 & k \\ 0 & 1\end{smallmatrix} \right)\tau_\ca^{-1}\right). $$
Étant données les relations~\eqref{eq:def-psi12} et~\eqref{eq:def-Fpsi}, la fonction
$$ t \mapsto F\Big(\tau\begin{pmatrix} 1 & t \\ 0 & 1\end{pmatrix}\tau_\ca^{-1}\Big) $$
est lisse, à support compact, et ne dépend que de la ligne inférieure de~$\tau$. Si~$\tau = \left(\begin{smallmatrix}\ast & \ast \\ \gamma & \delta\end{smallmatrix}\right)$, alors
$$ F\Big(\tau\begin{pmatrix} 1 & t \\ 0 & 1\end{pmatrix}\tau_\ca^{-1}\Big) = F\Big( \begin{pmatrix} \ast & \ast \\ \gamma w_\ca^{-1/2} & \gamma(t+\delta/\gamma) w_\ca^{1/2} \end{pmatrix}\Big). $$
La formule de Poisson fournit
$$ \sum_{k\in\ZZ} F\Big(\tau\begin{pmatrix} 1 & k \\ 0 & 1\end{pmatrix}\tau_\ca^{-1}\Big) = \e\Big(\frac{n\delta}{\gamma}\Big) \sum_{n\in\NN} G(\gamma, n), $$
où
$$ G(\gamma, n) = \int_\RR F\Big( \begin{pmatrix} \ast & \ast \\ \gamma w_\ca^{-\ud} & \gamma t w_\ca^{\ud} \end{pmatrix} \Big) \e(-nt)\dd t. $$
En utilisant la définition~\eqref{eq:def-Kloo}, nous obtenons finalement
\begin{equation}
U(h, \Psi) = \sum_{n\in\ZZ} \sum_{\gamma \in \cC(\infty, \ca)} S_{\infty \ca}(h, n ; \gamma) G(\gamma, n).\label{eq:Uhpsi-poisson}
\end{equation}

\subsubsection{Localisation et préparation des variables}

Nous rappelons que~$w_\ca \asymp_Q qd$. Par définition de~$F = F_{\Psi, Q}$, nous avons
$$ G(\gamma, n) = \int_\RR \Psi(tw_\ca) f\Big(\frac{Q(\gamma, \gamma t w_\ca)}{qw_\ca M}\Big)\e(-nt)\dd t. $$
Lorsque l'intégrant est non-nul, nous avons~$|t|\leq 2w_\ca^{-1}$ et~$\gamma\asymp_{Q, f} q(dM)^{\ud}$.  En intégrant par parties (\textit{cf.} le lemme 5.1 de~\cite{Toth}), il vient
\begin{equation}
G(\gamma, n) \ll_j (qd)^{j-1}n^{-j} \qquad (j\in\NN).\label{eq:majo-G}
\end{equation}
Cette majoration dépend aussi des constantes implicites dans~\eqref{eq:majo-der-f}~; cette dépendance ne sera pas explicitée afin d'alléger la notation.

Posons~$N_1 := qd(Mq)^\eta$. Dans le membre de droite de~\eqref{eq:Uhpsi-poisson}, nous isolons la contribution~$U_0$ (resp. $U_1$) provenant de~$n=0$ (resp. $|n|>N_1$). Les bornes~\eqref{eq:majotriv-kloosterman}, \eqref{eq:majogauss-kloosterman} et~\eqref{eq:majo-G} permettent d'écrire
\begin{equation}\label{majo-U0-U1}
\begin{aligned}
|U_0(h, \Psi)| {}&\ll_{\ee, D} (qd)^{-1}\ssum{\gamma\in\cC(\infty, \ca) \\ \gamma \asymp q(dM)^{1/2}} |S_{\infty\ca}(h, 0 ; \gamma)| \ll (Mqh)^\ee d^{-1}q^{-\ud}M^{\ud}, \\
|U_1(h, \Psi)| {}&\ll_{j, D} (qM)^{3/2}\Big(\frac{qd}N\Big)^{j-1} \ll_{\eta, D} (qM)^{-10}
\end{aligned}
\end{equation}
en choisissant~$j=j(\eta)$ suffisamment grand. Ces deux termes d'erreur sont bien de l'ordre du membre de droite de~\eqref{eq:majo-expo}.

D'autre part, la formule de Fa\`{a} di Bruno montre que la fonction~$G$ vérifie
$$ \frac{\partial^{k+\ell}}{\partial x^{\ell_1} \partial y^{\ell_2}}G(x, y)\Big|_{\substack{x=\gamma\\y=n}} \ll_{\ell_1, \ell_2} \gamma^{-\ell_1} (qd)^{-\ell_2-1}. $$
De même que~\eqref{eq:majo-G}, cette majoration dépend également des constantes implicites dans~\eqref{eq:majo-der-f}.

Nous introduisons une partition de l'unité pour la variable~$n$,
$$ G(\gamma, n) = \ssum{0\leq k \leq K} G_{2^k}(\gamma, n), $$
où~$K\leq 2+\log(N_1)/\log 2$, et pour~$1\leq N \leq N_1$, la fonction~$G_N(\gamma, n)$ est lisse par rapport aux deux variables, nulle en dehors de~$n\in[N/2, 2N]$ et satisfait
\begin{equation}
\frac{\partial^{k+\ell}}{\partial x^k \partial y^\ell}G_N(x, y)\Big|_{\substack{x=\gamma\\y=n}} \ll (qd)^{-1} \gamma^{-k} (\min\{qd, N\})^{-\ell} \ll (qM)^{O(\eta\ell)} (qd)^{-1}\gamma^{-k} N^{-\ell}. \label{eq:majo-GN-der}
\end{equation}
En accord avec cette décomposition, nous écrivons
\begin{equation}
\sum_{0<|n|\leq N_1} \sum_{\gamma\in\cC(\infty, \ca)} S_{\infty\ca}(h, n ; \gamma) G(\gamma, n) = \sum_{0\leq k \leq K} V_{2^k},\label{eq:decomp-unite}
\end{equation}
\begin{equation}
V_N = \sum_{N/2 \leq |n| \leq 2N} \sum_{\gamma\in\cC(\infty, \ca)} S_{\infty\ca}(h, n ; \gamma) G_N(\gamma, n).\label{eq:def-VN}
\end{equation}
Nous posons ensuite
$$ F(x, \xi) = \int_\RR G_N\Big(\frac{4\pi\sqrt{h|y|}}{x}, y\Big)\e(y\xi)\dd y, \quad G_N(\gamma, n) = \int_\RR F\Big(\frac{4\pi \sqrt{h|n|}}{\gamma}, \xi\Big)\e(-n\xi)\dd \xi. $$
L'intégrale définissant~$F$ est à support sur~$y\asymp N$, et lorsque~$F(x, \xi)\neq 0$, nous avons nécessairement~$x\asymp X := (hN/q^2dM)^{\ud}$.
La formule de Fa\`{a} di Bruno implique encore
$$ \partial_{k0}F(x, \xi) \ll_k (qM)^{O(\eta)} X^{-k} \frac{(qd)^{-1}N}{1+(N\xi)^2}. $$
Ici la constante implicite dans~$O(\eta)$ est indépendante de~$k$. Nous posons finalement
$$ \phi_\xi(x) = X(1+(N\xi)^2) qd (xN)^{-1} (qM)^{-\varpi} F(x, \xi) $$
pour un certain réel positif~$\varpi=O(\eta)$, de sorte que la fonction~$x\mapsto \phi_\xi(x)$ soit lisse, à support sur~$x\asymp X$ et satisfasse
\begin{equation}
\sup_{\xi\in\RR} \|\phi_\xi^{(j)}\| \ll_j X^{-j}. \label{eq:majo-derphi}
\end{equation}
Ici encore la borne dépend des constantes implicites dans~\eqref{eq:majo-der-f}. Cela ne sera plus rappelé dans la suite. Nous avons alors
\begin{equation}
V_N = 4\pi (qM)^{\varpi} d^{-\ud}M^{\ud}\int_\RR \frac{N}{1+(N\xi)^2} W_N(\xi) \dd \xi, \label{eq:VN-prep}
\end{equation}
où nous avons posé
\begin{equation}
W_N(\xi) := \sum_{N/2\leq |n| \leq 2N}a_n \sum_{\gamma\in\cC(\infty, \ca)} \frac{S^{(\xi)}_{\infty\ca}(h, n ; \gamma)}{\gamma} \phi_\xi\Big(\frac{4\pi\sqrt{h|n|}}{\gamma}\Big), \label{eq:def-WN}
\end{equation}
ainsi que~$a_n=\sqrt{|n|/N}$, et où nous avons intégré le facteur~$\e(-n\xi)$ dans la matrice d'échelle de~$\infty$ (ce qui est indiqué par la notation~$S^{(\xi)}_{\infty\ca}$).

\subsubsection{Utilisation de la formule de Kuznetsov}\label{sec:kuz}

Nous majorons~$W_N$ pour chaque~$\xi$. Nous omettrons la quantité~$\xi$ de la notation, et n'utiliserons de~$(a_n)$ et~$\phi$ que la borne~$|a_n|\leq 1$, les majorations~\eqref{eq:majo-derphi} et le fait que~$\phi(x)\neq 0$ entraîne~$x\asymp X$, où nous rappelons que~$X \asymp (hN/q^2dM)^{\ud}$.

Pour chaque~$n\in[N/2, 2N]$, nous appliquons la formule de Kuznetsov (lemme 4.5 de~\cite{D-Kuz}, avec~$\kappa=0$). Nous obtenons
\begin{align*}
\sum_{\gamma\in\cC(\infty, \ca)} \frac{S_{\infty\ca}(h, n ; \gamma)}{\gamma} \phi\Big(\frac{4\pi\sqrt{hn}}{\gamma}\Big) {}& = \cH^+_{h, n} + \cE^+_{h, n} + \cM^+_{h, n}, & (n>0)\\
\sum_{\gamma\in\cC(\infty, \ca)} \frac{S_{\infty\ca}(h, n ; \gamma)}{\gamma} \phi\Big(\frac{4\pi\sqrt{h|n|}}{\gamma}\Big) {}& = \cE^-_{h, n} + \cM^-_{h, n}, & (n<0)
\end{align*}
où
$$ \cM^+_{h, n} = \sum_{f\in \cB(q, \chi)} \frac{{\tilde \phi}(t_f)}{\cosh(\pi t_f)} (hn)^{\ud}\bar{\rho_{f\infty}(h)}\rho_{f\ca}(n), $$
et~$\cM^-_{h, n}$,~$\cE^\pm_{h, n}$, $\cH^+_{h, n}$ sont données par des expressions similaires. Ici, l'ensemble~$\cB(q, \chi)$ désigne une base orthonormée de formes de Maass paraboliques~$f$, chacune étant fonction propre du Laplacien hyperbolique, de valeur propre associée~$\lambda_f = \frac14+t_f^2$ et de coefficients de Fourier~$\rho_{f\ca}(n)$. Nous référons à la section 4.1.2 de~\cite{D-Kuz} pour les définitions précises et la normalisation. Nous avons~$t_f\in\RR\cup[-i\theta, i\theta]$, où nous rappelons que~$\theta\leq 7/64$ grâce à Kim et Sarnak~\cite{Kim}, et que la conjecture de Selberg-Ramanujan prédit que~$\theta=0$. La transformée~${\tilde\phi}$ dans l'expression ci-dessus est donnée par
$$ {\tilde\phi}(t) = \frac{2\pi i}{\sinh(\pi t)}\int_0^\infty (J_{2it}(x)-J_{-2it}(x))\phi(x)\frac{\dd x}{x} $$
où~$J_\nu(x)$ désigne la fonction de Bessel. La transformée~${\tilde\phi}$ satisfait les bornes énoncées au lemme~4.4 de~\cite{D-Kuz} (voir le lemme~2.4 de~\cite{Topacogullari} pour des bornes plus précises). Dans le cas présent, nous avons~$X\ll (qM)^{\eta/2}$, donc
$$ |{\tilde\phi}(t)| \ll \begin{cases} (qM)^{2\eta}(1+|t|)^{-3}, & t\in\RR, \\ (Mq)^{\eta/2} (q^2dM/hN)^{|t|}, & t\in[-i/4, i/4]. \end{cases} $$

La quantité~$\cE_{h, n}^{\pm}$ (resp.~$\cH_{h, n}^+$) correspond à la contribution des séries d'Eisenstein non holomorphes (resp. à la contribution des formes holomorphes de poids~$\geq 2$). Nous étudions en détail le cas~$\cM^+$, les autres termes étant traités de façon similaire.

Notre traitement diffère selon que nous prenons la moyenne sur~$h$ ou pas.

\subsubsection{Le cas~$(h, q)$ fixé}\label{sec:final-hfix}

Nous séparons dans~$\cM^+_{h, n}$ la contribution des fonctions~$f\in\cB(q, \chi)$ avec~$t_f\in\RR$, de celles avec~$t_f\in i\RR$. En accord avec cette décomposition, nous écrivons
\begin{equation}
W_N = \sum_{N/2\leq n \leq 2N} a_n \cM^+_{h, n} = \Mrg_{h, N} + \Mex_{h, N},\label{eq:dicho-cM}
\end{equation}
où la notation correspond à ``régulier'' et ``exceptionnel''. L'inégalité de Cauchy-Schwarz fournit
\begin{equation}
|\Mrg_{h, N}|\leq (\Mrg_h \Mrg_N)^{\ud},\label{eq:cauchy-reg}
\end{equation}
avec
\begin{align*}
\Mrg_h {}& := \ssum{f\in\cB(q, \chi) \\ t_f\in\RR}\frac{|{\tilde \phi}(t_f)|}{\cosh(\pi t_f)} h|\rho_{f\infty}(h)|^2, \\
\Mrg_N {}& := \ssum{f\in\cB(q, \chi) \\ t_f\in\RR} \frac{|{\tilde \phi}(t_f)|}{\cosh(\pi t_f)}\Big|\sum_{N/2\leq n \leq 2N} a_n n^{\ud}\rho_{f\ca}(n)\Big|^2.
\end{align*}
Pour majorer~$\Mrg_h$, nous faisons appel au lemme 2.7 de~\cite{Topacogullari}, soit
$$ \Mrg_h \ll_\ee (qhM)^\ee \Big\{1 + (qd, h)^\ud\frac{h^\ud}{qd^\ud}\Big\}. $$
Pour majorer~$\Mrg_N$, nous utilisons l'inégalité de grand crible (proposition 4.7 de~\cite{D-Kuz})
\begin{equation}
\Mrg_N \ll_\ee (qM)^\ee N \Big\{1 + \frac{N}{qd^\ud}\Big\}.\label{eq:majo-Mreg-N}
\end{equation}
Notre hypothèse~$h\ll q$ et~$N\leq N_1$ implique donc
\begin{equation}
\Mrg_{h, N} \ll_\eta (qM)^{O(\eta)} q^\ud d^{\frac34}.\label{eq:contrib-Mreg}
\end{equation}

Pour chaque~$h$, nous avons par l'inégalité de Cauchy-Schwarz
\begin{equation}
|\Mex_{h, N}| \ll (\Mex_h \Mex_N)^{\ud},\label{eq:majo-Mexc-CS}
\end{equation}
\begin{equation}\label{eq:def-CS-exc}
\Mex_h := \ssum{f\in\cB(q, \chi) \\ t_f\in i\RR}|{\tilde \phi}(t_f)|^2 h|\rho_{f\infty}(h)|^2, \quad
\Mex_N := \ssum{f\in\cB(q, \chi) \\ t_f\in i\RR}\Big|\sum_{N/2\leq n \leq 2N} a_n n^{\ud}\rho_{f\ca}(n)\Big|^2 .
\end{equation}
L'inégalité de grand crible fournit encore
\begin{equation}
\Mex_N \ll_\ee (qM)^\ee N \Big\{1 + \frac{N}{qd^\ud}\Big\}.\label{eq:majo-Mexc-N}
\end{equation}
Pour~$\Mex_h$, nous utilisons le lemme 2.9 de~\cite{Topacogullari},
$$ \Mex_h \ll_\eta (qhM)^{O(\eta)} \{(qd, h) M N^{-1}\}^{2\theta}. $$
Notre hypothèse~$N\leq N_1$ implique donc
$$ \Mex_{h, N} \ll_\eta (qhM)^{O(\eta)} d^{\frac14}(qd, h)^\theta M^\theta (qd)^{\ud-\theta}. $$
Le membre de droite de cette inégalité est supérieur à celui obtenu en~\eqref{eq:contrib-Mreg}. Nous obtenons donc au final
$$ \sum_{N/2\leq n \leq 2N} a_n \cM^+_{h, n} \ll_\eta (qhM)^{O(\eta)} d^{\frac14}(qd, h)^\theta M^\theta (qd)^{\ud-\theta}. $$
La même majoration est valable dans le cas de~$\cM^-_{h, n}$, tandis que les termes~$\cE^\pm$ et~$\cH^+$, sont majorés par une quantité de l'ordre du membre de droite de~\eqref{eq:contrib-Mreg}. Nous avons ainsi
$$ W_N \ll_\eta (qM)^{O(\eta)} d^{\frac14}(qd, h)^\theta M^\theta (qd)^{\ud-\theta}. $$

Nous insérons cela dans~\eqref{eq:VN-prep} puis \eqref{eq:decomp-unite}, ce qui fournit, avec les majorations~\eqref{majo-U0-U1} et en choisissant~$\eta>0$ arbitrairement petit,
$$ U(h, \Psi) \ll_\ee (qM)^{\ee} d^{-\frac14}(qd, h)^\theta M^\theta (qd)^{\ud-\theta}. $$
Nous reportons cela dans~\eqref{eq:def-TQ}, en utilisant le point~(i) du Lemme~\ref{lemme:pointe} pour majorer le cardinal de la somme sur~$\sigma$. Nous obtenons
$$ T_Q(h, \Psi) \ll_\ee (qM)^{\ee} d^{\frac34}(qd, h)^\theta M^\theta (qd)^{\ud-\theta}, $$
ce qui fournit la majoration~\eqref{eq:majo-expo} grâce à~\eqref{eq:SQ-TQ},~\eqref{eq:SMq-Shpsi} et~\eqref{eq:mobius} successivement.

\subsubsection{Majoration en moyenne sur~$h$}\label{sec:final-hmoy}

Nous justifions dans cette section la majoration~\eqref{eq:somme-expo-moyh}. Lorsque~$H\leq qd^\ud$, nous nous contentons de prendre la moyenne sur~$h$ de l'estimation~\eqref{eq:majo-expo} établie dans les sections précédentes. Nous supposons donc dorénavant que~$H>qd^\ud$.

Posons~$(c_h) \in \CC^\NN$, $|c_h|\leq1$ une suite telle que
$$ \Big|\sum_{(m, \Omega)\in \cD} f\Big(\frac{m}{M}\Big)\e\Big(\frac{h\Omega}{mq}\Big)\Big| = c_h \sum_{(m, \Omega)\in\cD} f\Big(\frac{m}{M}\Big)\e\Big(\frac{h\Omega}{mq} \Big). $$
Les coefficients~$(c_h)$ dépendent au plus de~$(h, q, d, \lambda, \omega, M, f)$.

Rappelant la définition~\eqref{eq:dicho-cM}, il nous suffira d'établir les majorations
\begin{equation}
\Mrg_{H, N} := \frac1H\sum_{H<h \leq 2H} c_h\Mrg_{h, N} \ll_\eta (qHM)^{O(\eta)} d^{\frac34} q^\ud,\label{eq:majo-Mreg-moyh}
\end{equation}
\begin{equation}
\Mex_{H, N} := \frac1H\sum_{H<h \leq 2H} c_h\Mex_{h, N} \ll_\eta (qHM)^{O(\eta)} d^{\frac14} M^\theta (qd)^{\ud-\theta}.\label{eq:majo-Mexc-moyh}
\end{equation}
Les bornes~\eqref{eq:majo-Mreg-N} et~\eqref{eq:majo-Mexc-N} sont toujours valables en moyenne sur~$h$, puisqu'elles ne dépendent pas de~$h$.

Dans le cas de \eqref{eq:majo-Mreg-moyh}, un raisonnement similaire à~\eqref{eq:cauchy-reg} nous ramène à étudier
$$ \Mrg_H := H^{-2} \ssum{f\in\cB(q, \chi) \\ t_f\in \RR}\frac{|{\tilde \phi}(t_f)|}{\cosh(\pi t_f)} \Big|\sum_{H < h \leq 2H} c_h \sqrt{h}\rho_{f\infty}(h)\Big|^2. $$
L'inégalité de grand crible fournit
$$ \Mrg_H \ll_\ee (qHM)^\ee H^{-1}\Big\{1 + \frac{H}{qd^\ud}\Big\} \ll_\ee (qHM)^\ee, $$
d'où l'on déduit la borne~\eqref{eq:majo-Mreg-moyh}.

Concernant~\eqref{eq:majo-Mexc-moyh}, en raisonnant de façon similaire à~\eqref{eq:majo-Mexc-CS}, nous nous ramenons à étudier
$$ \Mex_{H} := H^{-2} \ssum{f\in\cB(q, \chi) \\ t_f\in i\RR}|{\tilde \phi}(t_f)|^2 \Big|\sum_{H < h \leq 2H} c_h \sqrt{h}\rho_{f\infty}(h)\Big|^2. $$
Nous appliquons l'inégalité de grand crible pour le spectre exceptionnel, lemme 4.8 de~\cite{D-Kuz}, en tirant parti de la borne de Kim-Sarnak (voir la remarque qui précède la section 4.3 de~\cite{D-Kuz}). Nous obtenons
\begin{align*}
\Mex_H {}& \ll_\eta H^{-1}(qM)^{O(\eta)} \Big(1+\Big(\frac{qM}N\Big)^{2\theta}\Big)\Big(1+d^{\ud}\Big(\frac{H}{qd}\Big)^{1-2\theta}\Big) \\
{}& \ll_\eta (qM)^{O(\eta)} (q^2d)^{2\theta-\ud}\Big(\frac MN\Big)^{2\theta}.
\end{align*}
Cela suffit pleinement à démontrer l'estimation~\eqref{eq:majo-Mexc-moyh}. La formule~\eqref{eq:somme-expo-moyh} en est déduite de façon identique au cas~$h$ fixé.

\subsubsection{Majoration en moyenne sur~$h$ et sur~$q$}\label{sec:final-hqmoy}

Supposons maintenant que~$d=1$, et que~$c_h = \e(th)\1_{h\in\cI}$ pour un certain intervalle~$\cI\subset[H, 2H]$. Nous suivons les arguments des sections précédentes, en encodant le facteur~$\e(th)$ dans la matrice d'échelle de~$\infty$, ce qui nous ramène à l'estimation de
$$ \Mex_{H}(q) = H^{-2} \ssum{f\in\cB(q, \1) \\ t_f\in i\RR}|{\tilde \phi}(t_f)|^2 \Big|\sum_{h\in\cI} \sqrt{h}\rho_{f\infty}(h)\Big|^2. $$
Nous sommons cela sur~$q\in[Q, 2Q]$, et utilisons l'inégalité de grand crible pondérée de Deshouillers-Iwaniec, Theorem 7 de~\cite{DI}. Nous obtenons
$$ \frac1Q \sum_{Q<q\leq 2Q} \Mex_{H}(q) \ll_{\ee, \eta} M^{O(\eta)} H^{-1+\ee} \Big\{1 + \frac{H}{Q} + \Big(\frac{M}{N}\Big)^{2\theta}\Big\}. $$
Avec~$2\theta$ remplacé par~$\ud$, cela découle directement du Theorem~7 de~\cite{DI}. L'équation qui précède est aisément justifiée en notant qu'à la conclusion de la preuve du Theorem 7, page 278 de~\cite{DI}, la quantité~$\sqrt{Y/Y_1}$ peut être remplacée par~$(Y/Y_1)^{2\theta}$. La conclusion de la preuve suit de façon identique au cas~$h$ fixé.

\begin{remarques} \hfill
\begin{itemize}
\item Lorsque~$H$ est grand, le terme d'erreur que nous obtenons est légèrement meilleur que celui annoncé en~\eqref{eq:somme-expo-moyh}. Cela n'a pas d'influence pour l'application que nous considérons ici.
\item Les facteurs~$h$ et~$H$ apparaissant aux premiers termes des membres de droite de~\eqref{eq:majo-expo}-\eqref{eq:somme-expo-moyqh} peuvent être améliorés en opérant une intégration par parties au lieu de l'approximation triviale~\eqref{eq:approx-triv-exp}.
\end{itemize}
\end{remarques}

\subsection{Démonstration du \texorpdfstring{Lemme~\ref{lemme:type1}}{Lemme type-1}}

Du Lemme~\ref{lemme:somme-expo}, nous déduisons par une technique standard d'analyse de Fourier l'estimation
\begin{equation}
\begin{aligned}
P_f(M ; q, r, {}& d, \lambda, \omega, \alpha, \beta) \\ & = (\beta-\alpha)P_f(M; q, r, d, \lambda, \omega, 0, 1) + O_{\ee, D, f}((qM)^\ee d^{\frac34} (qd)^{\ud-\theta} M^{\ud+\theta}).
\end{aligned}\label{eq:estim-interm-Pf}
\end{equation}
Nous omettons les détails, qui sont similaires aux pages 179 et 180 de~\cite{Iwaniec}. La seule différence avec notre traitement tient aux termes supplémentaires~$h$ et~$H$ dans les membres de droite de~\eqref{eq:majo-expo} et~\eqref{eq:somme-expo-moyh}, qui imposent le choix~$\Delta = (q+M^\ud)^{-1}$ dans l'argument d'Iwaniec. Cela induit un terme d'erreur supplémentaire de l'ordre de
$$ \Delta^{-1} + M\Delta \ll q + M^\ud \ll q^{\ud-\theta}M^{\ud+\theta}, $$
qui est acceptable.

Nous nous concentrons donc sur le traitement du terme principal. Nous aurons besoin du lemme suivant.
\begin{lemme}\label{lemme:rD}
Soit~$x\in\RR$ avec $x\geq 1$, $D\in\ZZ$ qui n'est pas un carré d'entier, $(q, d)\in\NN^2$ avec~$q\geq 1$, $(q, 2D)=1$, $d|q$ et~$\lambda\mod{d}$ avec~$(\lambda, d)=1$. Notons~$\chi_D = (\frac D\cdot)$ le symbole de Kronecker, et~$\varkappa_D(n) := (1\ast \chi_D)(n)$. Alors
$$ \ssum{n\leq x \\ (n, q)=1 \\ n\equiv \lambda\mod{d}} \varkappa_D(n) = \frac{x}{\vphi(d)}\frac{\vphi(q)}{q}\prod_{p|q}\Big(1-\frac{\chi_D(p)}p\Big) L(1, \chi_D) +O_{\ee, D}(x^{\ud}q^\ee). $$
\end{lemme}
\begin{proof}
Cela découle facilement du principe de l'hyperbole de Dirichlet.
\end{proof}
Rappelons que pour~$p\nmid 2D$, nous avons~$\rhoQ(p) = 1+(\frac Dp) = \varkappa_D(p)$. Nous écrivons~$\rhoQ = \varkappa_D \ast h_D$, de sorte que la fonction~$h_D$ vérifie~$\sum_\ell |h_D(\ell)|\ell^{-\ud-\ee} \ll_{\ee, D} 1$. Lorsque~$M\geq 1$ et~$(\lambda, d)=1$, nous en déduisons, à l'aide du Lemme~\ref{lemme:rD} et d'une intégration par parties,
\begin{align*}
\ssum{(m, q)=1 \\ m\equiv \lambda\mod{d}} f\Big(\frac mM\Big) \rhoQ(m) {}& = \ssum{(\ell, q)=1 \\ \ell \ll M} h_D(\ell) \ssum{(n, q)=1 \\ n\equiv \lambda{\bar{\ell}} \mod{d}} f\Big(\frac{n\ell}M\Big) \varkappa_D(n) \\
{}& = \frac1{\vphi(d)}\frac{\vphi(q)}qL(1, \chi_D) M{\hat f}(0)\ssum{(\ell, q)=1} \frac{h_D(\ell)}{\ell} + O_{\ee, D, f}(q^\ee M^{\ud+\ee}). \numberthis\label{eq:estim-rhom}
\end{align*}
Notons
$$ C_D := L(1, \chi_D)\sum_{\ell\geq 1} \frac{h_D(\ell)}\ell = \sum_{\ell\geq 1} \frac{(\rhoQ \ast \mu)(\ell)}\ell. $$
Nous obtenons
\begin{equation}
L(1, \chi_D) \frac{\vphi(q)}q \sum_{(\ell, q)=1} \frac{h_D(\ell)}{\ell} = C_D \prod_{p|q}\Big(1+\frac1p\Big)^{-1}.\label{eq:cte-q}
\end{equation}

Nous revenons maintenant à l'estimation du terme principal du membre de droite de~\eqref{eq:estim-interm-Pf}. Le théorème des restes chinois et les relations~\eqref{eq:estim-rhom} et~\eqref{eq:cte-q} avec~$q$ remplacé par~$qr$ fournissent
\begin{align*}
P_f(M; q, r, d, \lambda, \omega, 0, 1) {}& = \ssum{(m, qr)=1 \\ m\equiv \lambda\mod{d}} f\Big(\frac mM\Big) \ssum{\Omega \mod{qm} \\ \Omega^2 \equiv D \mod{qm} \\ \Omega \equiv \omega\mod{d}} 1 \\
{}& = \rhoQ_{\omega, d}(q) \ssum{(m, qr)=1 \\ m\equiv \lambda\mod{d}} f\Big(\frac mM\Big) \rhoQ(m) \\
{}& = C_D \prod_{p|qr}\Big(1+\frac1p\Big)^{-1} \frac{\rhoQ_{\omega, d}(q)}{\vphi(d)} M{\hat f}(0) + O_{\ee, D, f}(M^{\ud+\ee}q^\ee),
\end{align*}
où nous avons noté, pour tout~$\omega\mod{d}$ avec~$\omega^2\equiv D\mod{d}$,
$$ \rhoQ_{\omega, d}(q) = \ssum{\Omega \mod{q} \\ \Omega^2 \equiv D \mod{q} \\ \Omega \equiv \omega\mod{d}} 1. $$
Il est aisé de voir que~$\rhoQ_{\omega, d}(q) = \rhoQ(q)$ si~$d=1$, et pour tout~$p\nmid 2D$, $1\leq \delta \leq \nu$, $\rhoQ_{\omega, p^\delta}(p^\nu) = 1$ par le lemme de Hensel. Nous en déduisons que~$\rhoQ_{\omega, d}(q) = \rhoQ(q/(q, d^\infty))$ indépendamment de~$\omega$. Cela conclut la démonstration du Lemme~\ref{lemme:type1}.

\subsection{Démonstration de la \texorpdfstring{Proposition~\ref{prop:equidistrib}}{Proposition equidistrib}}

\subsubsection{Première réduction}\label{sec:premiere-reduction}

Nous remarquons tout d'abord que la borne triviale~$x^{2+\ee}/M$ pour membre de gauche de~\eqref{eq:equidistrib} nous permet de supposer sans perte de généralité que~$x\geq M$.

Dans le but de simplifier la preuve de la Proposition~\ref{prop:equidistrib}, nous nous ramenons à supposer que la suite~$(b_n)$ est à support sur les entiers impairs premiers avec~$D$. Supposons donc dans un premier temps que l'estimation~\eqref{eq:equidistrib} est valable pour de telles suites. Notant
\begin{equation}
r_D(x; q) := \ssum{k\in\NN \\ q | k^2 - D} V\Big(\frac kx\Big) - x{\hat V}(0)\frac{\rhoQ(q)}{q},\label{eq:def-rD}
\end{equation}
nous avons, grâce à l'inégalité de Cauchy-Schwarz et la majoration~$\ssum{1\leq v\leq 2N \\ v|(2D)^\infty} 1 \ll x^\ee$,
\begin{align*}
\sum_{M < m \leq 2M} \Big|\ssum{N< n \leq 2N \\ (n, m)=1} b_n r_D(x ; mn)\Big|^2 
={}& \sum_{M< m \leq 2M} \Big|\ssum{1\leq v\leq 2N \\ v|(2D)^\infty} \ssum{N/v< n \leq 2N/v \\ (n, 2Dm)=1} b_{vn} r_D(x ; vmn)\Big|^2 \\
\ll_\ee{}& x^\ee \sum_{M< m \leq 2M} \ssum{1\leq v\leq 2N \\ v|(2D)^\infty} \Big| \ssum{N/v< n \leq 2N/v \\ (n, 2Dm)=1} b_{vn} r_D(x ; vmn)\Big|^2 \\
\ll_\ee {}& x^\ee \ssum{v\leq 2N \\ v|(2D)^\infty} \sum_{vM< m \leq 2vM} \Big| \ssum{N/v< n \leq 2N/v \\ (n, 2Dm)=1} b_{vn} r_D(x ; mn)\Big|^2.
\end{align*}
La majoration~\eqref{eq:equidistrib} appliquée pour chaque~$v$ au membre de droite fournit l'estimation voulue.

Nous supposons donc dans ce qui suit que la suite~$(b_n)$ est à support sur les entiers~$n$ tels que~$(n, 2D)=1$.

\subsubsection{Interprétation d'une congruence}

Nous suivons les arguments des pages 180-183 de~\cite{Iwaniec}. Pour cela nous devons modifier la construction de la classe~$c\mod{[n_1, n_2]}$, page 183 de~\cite{Iwaniec}, en raison du fait que dans notre situation,~$(b_n)$ n'est pas supposée à support sur les entiers sans facteur carré.
\begin{lemme}\label{lemme:congruence}
Soient~$m, n_1, n_2, \ell_1, \ell_2\geq 1$ donnés, avec~$(2mD, n_1n_2)=1$. Définissons
$$ d=(n_1, n_2)/(n_1, n_2, \ell_1-\ell_2), $$
et supposons que
\begin{equation}
(m(\ell_1-\ell_2))^2 \equiv 4D \mod{d}.\label{eq:cond-congr-m}
\end{equation}
Alors il existe~$c\in\ZZ$, avec~$0\leq c <[n_1, n_2]$, tel que les ensembles
$$ \cD_1 = \Big\{v\in\ZZ\cap[0, m\mc{[} :\ \begin{array}{l} v^2\equiv D\mod{m} \\ (m\ell_j+v)^2 \equiv D \mod{n_j} \quad (j\in\{1, 2\}) \end{array} \Big\} $$
et
$$ \cD_2 = \Big\{\Omega\in\ZZ\cap[cm, (c+1)m\mc{[} :\ \begin{array}{l} \Omega^2\equiv D\mod{m[n_1, n_2]} \\ \Omega \equiv m(c-\ud(\ell_1+\ell_2))\mod{d} \end{array}  \Big\} $$
soient en bijection.
\end{lemme}

\begin{remarque}
Les ensembles~$\cD_1$ et~$\cD_2$ sont vides si la condition~\eqref{eq:cond-congr-m} n'est pas satisfaite.
\end{remarque}

\begin{proof}
Notons
$$ n_j = \prod_p p^{\nu_j(p)} \qquad (j\in\{1, 2\}), $$
Nous définissons~$c\in\ZZ$, $0\leq c < [n_1, n_2]$ comme l'unique entier satisfaisant pour tout~$p$,
$$ c \equiv \begin{cases} \ell_1 \mod{p^{\nu_1(p)}}, & \text{si }\nu_1(p)\geq \nu_2(p), \\ \ell_2 \mod{p^{\nu_2(p)}} & \text{sinon.} \end{cases} $$
À tout~$v\in\ZZ\cap[0, m\mc{[}$, nous associons~$\Omega(v) = cm+v \in [cm, (m+1)c\mc{[}$. Cette application est bijective, et il nous suffit de vérifier que~$\Omega(\cD_1)=\cD_2$. Supposons~$v\in\cD_1$, et soit~$\Omega = \Omega(v)$. Puisque~$(m, [n_1, n_2])=1$, il suffit de vérifier la congruence~$\Omega^2\equiv D$ modulo~$m$ et~$[n_1, n_2]$, séparément. Nous avons~$\Omega \equiv v\mod{m}$, ce qui fournit bien~$\Omega^2\equiv D\mod{m}$. Pour tout~$p$, nous avons
$$ \Omega \equiv \ell_j m + v \mod{p^{\nu_j(p)}}, $$
avec~$j=1$ si~$\nu_1(p)\geq\nu_2(p)$, et~$j=2$ sinon. Nous obtenons dans les deux cas~$\Omega^2\equiv D \mod{p^{\nu_j(p)}}$, donc~$\Omega^2\equiv D\mod{[n_1, n_2]}$. La condition~$\Omega\equiv  m(c-\ud(\ell_1+\ell_2))\mod{d}$ découle facilement du fait que
$$ (m\ell_1+v)^2 \equiv (m\ell_2+v)^2\mod{(n_1, n_2)}. $$
Supposons ensuite~$\Omega\in\cD_2$ donné, et posons~$v = \Omega - mc$. La relation~$v^2\equiv D\mod{m}$ est alors immédiate. Soit ensuite~$p$ fixé, notons~$\nu_j = \nu_j(p)$ et supposons~$\nu_1\geq \nu_2$ (le cas complémentaire étant traité de façon identique). Nous avons alors
$$ c\equiv \ell_1 \mod{p^{\nu_1}}, \qquad \Omega^2 \equiv D \mod{p^{\nu_1}}, $$
ce qui fournit directement la relation~$(m\ell_1+v)^2\equiv D \mod{p^{\nu_1}}$. Nous avons par ailleurs
$$ (m\ell_2 + v)^2 \equiv \Omega^2 - 2m(\ell_1-\ell_2)\Omega + (m(\ell_1-\ell_2))^2 \mod{p^{\nu_2}}. $$
Par hypothèse, $\Omega^2 \equiv D \mod{p^{\nu_2}}$. Ensuite,
$$ \Omega \equiv m(c-\tfrac12(\ell_1+\ell_2)) \equiv \tfrac12 m (\ell_1-\ell_2) \mod{\tfrac{p^{\nu_2}}{(p^{\nu_2}, \ell_1-\ell_2)}}, $$
ce qui fournit
$$ 2m(\ell_1-\ell_2)  \Omega \equiv (m(\ell_1-\ell_2))^2 \mod{p^{\nu_2}}. $$
Nous en déduisons~$(m\ell_2+v)^2\equiv D \mod{p^{\nu_2}}$. Nous avons donc bien obtenu~$v\in\cD_1$.
\end{proof}

\subsubsection{Méthode de dispersion}

Nous développons le carré dans le membre de gauche de~\eqref{eq:equidistrib}. En accord avec~\cite{Iwaniec}, nous posons
$$ Y(m) := \ssum{N< n \leq 2N \\ (n, m)=1} b_n\frac{\rhoQ(n)}{n}. $$
Fixons aussi une fonction lisse~$f:\RR\to\RR$ telle que~$\1_{1\leq t \leq 2} \leq f(t) \leq \1_{1/2 \leq t \leq 3}$. Enfin, nous rappelons la notation~\eqref{eq:def-rD}. Le membre de gauche de~\eqref{eq:equidistrib} est majoré par
\begin{equation}
\begin{aligned}
{}& \sum_m f\Big(\frac mM\Big)\Big|\ssum{N< n \leq 2N \\ (n, m)=1} b_n r_D(x ; mn) \Big|^2 \\
= {}& \sum_m f\Big(\frac mM\Big)\Big|\ssum{N< n \leq 2N \\ (n, m)=1} b_n \ssum{k\in\NN \\ q | k^2 - D} V\Big(\frac kx\Big) - x{\hat V}(0)\frac{\rhoQ(q)}{q}r_D(x ; mn) \Big|^2  \\
= {}& \sum_m f\Big(\frac mM\Big)\Big|\ssum{0\leq v < m \\ v^2 \equiv D \mod{m}} \Big( \ssum{N< n \leq 2N \\ (n, m)=1} b_n \ssum{k\in\NN \\ k \equiv v \mod{m} \\ k^2 \equiv D \mod{n}} V\Big(\frac kx\Big) - x{\hat V}(0)\frac{\rhoQ(n)}{mn} \Big) \Big|^2 \\
\leq {}& \sum_m f\Big(\frac mM\Big) \rhoQ(m) \ssum{0\leq v < m \\ v^2 \equiv D \mod{m}} \Big|\ssum{N< n \leq 2N \\ (n, m)=1} b_n \ssum{k\in\NN \\ k \equiv v \mod{m} \\ k^2 \equiv D \mod{n}} V\Big(\frac kx\Big) - x{\hat V}(0)\frac{\rhoQ(n)}{mn} \Big|^2 \\
\ll_\ee {}&  x^\ee \big( S_1 - 2x\bar{{\hat V}(0)}\Re S_2 + |x{\hat V}(0)|^2 S_3 \big),\label{eq:dispersion}
\end{aligned}
\end{equation}
avec
$$ S_j = \sum_m f\Big(\frac mM\Big) \ssum{0\leq v < m \\ v^2 \equiv D \mod{m}} T_j(m), $$
et
$$ \begin{aligned} &\quad T_1(m) = \ssumm{N < n_1, n_2 \leq 2N \\ (n_1n_2, m)=1} b_{n_1}\bar{b_{n_2}} \ssumm{k_1, k_2 \in\NN \\ k_j \equiv v \mod{m} \\ k_j^2 \equiv D \mod{n_j}} V\Big(\frac{k_1}x\Big)\bar{V\Big(\frac{k_2}x\Big)}, \\
T_2(m) {}&= \frac{\bar{Y(m)}}{m} \ssum{N < n \leq 2N \\ (n, m)=1} b_n \ssum{k\in\NN \\ k\equiv v\mod{m} \\ k^2 \equiv D\mod{n}} V\Big(\frac kx\Big), \qquad\quad T_3(m) = \Big(\frac{Y(m)}{m}\Big)^2. \end{aligned} $$

\subsubsection{Estimation de~$S_3$}

Nous avons
$$ S_3 = \frac1{M^2}\ssumm{N < n_1, n_2 \leq 2N} b_{n_1}\bar{b_{n_2}} \frac{\rhoQ(n_1)\rhoQ(n_2)}{n_1n_2} \ssum{(m, n_1n_2)=1} \frac{M^2}{m^2}f\Big(\frac mM\Big) \rhoQ(m). $$
Notant~$g_1(t) = t^{-2}f(t)$, la somme en~$m$ du membre de droite vaut
$$ P_{g_1}(M ; 1, n_1n_2, 1, 1, 1, 0, 1). $$
Nous obtenons donc
\begin{equation}
S_3 = P_3 + O_{\ee, D}(x^\ee M^{-\frac32+\theta}),\label{eq:estim-S3}
\end{equation}
avec
$$ P_3 = C_D M^{-1}\Big(\int_\RR t^{-2}f(t)\dd t\Big) \ssumm{N < n_1, n_2 \leq 2N} b_{n_1}\bar{b_{n_2}} A(n_1n_2)\frac{\rhoQ(n_1)\rhoQ(n_2)}{n_1n_2}. $$

\subsubsection{Estimation de~$S_2$}

Nous avons
$$ S_2 = \ssumm{N < n_1, n_2 \leq 2N} b_{n_1}\bar{b_{n_2}} \frac{\rhoQ(n_2)}{n_2} \ssum{(m, n_1n_2)=1} \frac1m f\Big(\frac mM\Big) \ssum{0\leq v < m \\ v^2\equiv D\mod{m}} \ssum{k\in\NN \\ k \equiv v\mod{m} \\ k^2\equiv D \mod{n_1}} V\Big(\frac kx\Big). $$
Nous écrivons~$k=m\ell+v$ avec~$\ell\geq 0$ et $\ell\ll x/m$. Nous avons alors
$$ V\Big(\frac{m\ell+v}x\Big) = V\Big(\frac{m\ell}x\Big) + O\Big(\frac{m}{x}\Big), $$
ce qui fournit, de même que dans~\cite[formule~(11)]{Iwaniec}, l'approximation~$S_2 = S_2' + O(x^\ee)$ avec
$$ S_2' = \ssumm{N < n_1, n_2 \leq 2N} b_{n_1}\bar{b_{n_2}} \frac{\rhoQ(n_2)}{n_2} \ssum{(m, n_1n_2)=1} \frac1m f\Big(\frac mM\Big)\!\!\! \ssumm{\ell\geq 0 \\ 0\leq v < m \\ v^2\equiv D\mod{m} \\ (m\ell+v)^2\equiv D\mod{n_1}} \!\!\! V\Big(\frac{m\ell}x\Big). $$
Les supports de~$f$ et~$V$ impliquent que les entiers~$\ell$ fournissant une contribution non nulle à~$S_2'$ proviennent d'un intervalle d'entiers~$I$ tels que~$\ell\asymp x/M$ pour tout~$\ell\in I$. Pour tout~$n_2$ avec~$\rhoQ(n_2)\neq 0$, nous posons~$n_2'=n_2/(n_2, n_1^\infty)$. Soit~$c\in\NN\cap[0,n_1\mc{[}$ l'unique entier satisfaisant~$c\equiv \ell\mod{n_1}$. Nous avons une bijection
\begin{align*}
\Big\{v \in \NN\cap[0, m\mc{[}:\ {}&\begin{array}{l} v^2\equiv D \mod{m} \\ (m\ell+v)^2 \equiv D \mod{n_1} \end{array} \Big\} \\ &\longrightarrow \Big\{\Omega \in\NN\cap[0, mn_1\mc{[}:\ \begin{array}{l} \Omega^2 \equiv D \mod{mn_1} \\ cm \leq \Omega < (c+1)m \end{array} \Big\}
\end{align*}
donnée par~$v \mapsto mc + v$. Ainsi,
$$ S_2' = \frac1M\ssumm{N < n_1, n_2 \leq 2N} b_{n_1}\bar{b_{n_2}} \frac{\rhoQ(n_2)}{n_2\rhoQ(n_2/(n_2, n_1^\infty))}  \sum_{\ell\in I} \ssum{(m, n_1n_2)=1} g_{2, \ell}\Big(\frac mM\Big) \ssum{\Omega \in \NN \\ \Omega^2 \equiv D \mod{mn_1} \\ cm \leq \Omega < (c+1)m}1. $$
avec~$g_{2, \ell}(t) = t^{-1}f(t) V(t\ell M/x)$, qui satisfait l'hypothèse~\eqref{eq:majo-der-f}. La somme sur~$(m, \Omega)$ est exactement~$P_{g_{2, \ell}}(M ; n_1, n_2', 1, 1, 1, \frac{c}{n_1}, \frac{c+1}{n_1})$, le Lemme~\ref{lemme:type1} fournit donc
$$ S_2' = P_2 + O_{\ee, D}(x^\ee N^{-\frac32-\theta}M^{-\ud+\theta}), $$
avec
$$ P_2 = C_D \ssumm{N < n_1, n_2 \leq 2N} b_{n_1}\bar{b_{n_2}} \frac{\rhoQ(n_1)\rhoQ(n_2)}{n_1n_2}A(n_1n_2) \int t^{-1}f(t)\sum_{\ell\in\ZZ} V\Big(\frac{\ell t M}x\Big)\dd t. $$
Uniformément pour~$t\in\supp f$, nous utilisons
$$ \sum_{\ell\in\ZZ} V\Big(\frac{\ell t M}x\Big) = \frac{x}{Mt}{\hat V}(0) + O(1), $$
ce qui fournit~$P_2 = x{\hat V}(0) P_3 + O_{\ee, D}(x^\ee)$, et finalement
\begin{equation}
S_2 = x{\hat V}(0) P_3 + O_{\ee, D}(x^\ee\{1 + N^{-\frac32-\theta}M^{-\ud+\theta}\}).\label{eq:estim-S2}
\end{equation}

\subsubsection{Estimation de~$S_1$ et conclusion}

Dans la somme~$S_1$, nous posons~$k_j = m\ell_j+v$ avec~$\ell_j\geq 0$, ainsi
$$ S_1 = \ssumm{N< n_1, n_2 \leq 2N} b_{n_1}\bar{b_{n_2}}~\ssumm{\ell_1, \ell_2\geq 0} \ssum{(m, n_1n_2)=1} f\Big(\frac mM\Big) \!\!\! \ssum{0\leq v < m \\ v^2\equiv D \mod{m} \\ (m\ell_j+v)^2\equiv D \mod{n_j}} \!\!\! V\Big(\frac{m\ell_1+v}x\Big)\bar{V\Big(\frac{m\ell_2+v}x\Big)}. $$
Nous remplaçons le produit~$V(\dotsc)\bar{V(\dotsc)}$ par~$V(m\ell_1/x)\bar{V(m\ell_2/x)}$. L'erreur induite dans~$S_1$ est~$O_{\ee, D}(x^{1+\ee})$, de sorte que~$S_1=S_1'+O_{\ee, D}(x^{1+\ee})$, avec
\begin{equation}
\begin{aligned}
S_1' = \ssumm{N< n_1, n_2 \leq 2N} {}& b_{n_1}\bar{b_{n_2}}~\ssumm{\ell_1, \ell_2\geq 0} \ssum{(m, n_1n_2)=1}  f\Big(\frac mM\Big) \times \\
& \times V\Big(\frac{m\ell_1}x\Big)\bar{V\Big(\frac{m\ell_2}x\Big)} \!\!\! \ssum{0\leq v < m \\ v^2\equiv D \mod{m} \\ (m\ell_j+v)^2\equiv D \mod{n_j}} \!\!\! 1
\end{aligned}\label{eq:reexpr-S1}
\end{equation}
Pour chaque~$(n_1, n_2, \ell_1, \ell_2)$, la somme sur~$v$ est exprimée au moyen du Lemme~\ref{lemme:congruence}. Nous posons~$q = [n_1, n_2]$, $d=(n_1,n_2)/(n_1, n_2, \ell_1-\ell_2)$, et
$$ \cL = \{\lambda\mod{d}:\ (\lambda(\ell_1-\ell_2))^2 \equiv 4D \mod{d}\}. $$
Puisque~$(d, 2D)=1$, nous avons~$\cL=\varnothing$ si~$(\ell_1-\ell_2, d)>1$, et~$|\cL|=\rhoQ(d)$ sinon. Nous supposons donc que la somme sur~$(\ell_1, \ell_2)$ est restreinte à~$(\ell_1-\ell_2, d)=1$. La somme sur~$(m, v)$ dans le membre de droite de~\eqref{eq:reexpr-S1} vaut
$$ \sum_{\lambda\in \cL} P_{g_3}\Big(M ; q, d, \lambda, \omega_\lambda, \frac cq, \frac{c+1}q\Big) \qquad (\omega_\lambda = \lambda(c-\tfrac12(\ell_1+\ell_2))), $$
avec~$g_3(t) = f(t)V(t\ell_1 M/x)\bar{V(t\ell_2 M/x)}$. Puisque~$|\cL|=\rhoQ(d)$ et~$\rhoQ(d)\rhoQ(q/(q, d^\infty)) = \rhoQ(q)$, le Lemme~\ref{lemme:type1} fournit
$$ S_1' = P_1 + O_{\ee, D}\Big(x^{1+\ee} + x^{2+\ee}\Big(\frac{N^2}M\Big)^{\tfrac32-\theta}\Big), $$
avec
\begin{equation}
P_1 = C_D M \ssumm{N< n_1, n_2 \leq 2N} b_{n_1}\bar{b_{n_2}} \ssumm{\ell_1, \ell_2\geq 0 \\ (\ell_1-\ell_2, d)=1} \frac{\rhoQ(q)}{q} \frac{A(q)}{\vphi(d)} \int_\RR f(t)V\Big(\frac{t\ell_1 M}x\Big)\bar{V\Big(\frac{t\ell_2 M}x\Big)}\dd t.\label{eq:expr-P1-l1l2}
\end{equation}
Notons temporairement~$n_0 = (n_1, n_2)$. Rappelons que~$d=n_0/(n_0, \ell_1-\ell_2)$. Pour~$X\gg 1$, nous avons
\begin{equation}\label{eq:somme-l1l2}
\begin{aligned}
\ssumm{\ell_1, \ell_2 \in \NN \\ (\ell_1-\ell_2, n_0) = n_0/d \\ (\ell_1-\ell_2, d)=1} V\Big(\frac{\ell_1}X\Big)\bar{V\Big(\frac{\ell_2}X\Big)}
= {}& \1_{(d, n_0/d)=1} \sum_{\ell\in\NN} V\Big(\frac\ell X\Big) \ssum{k\in\ZZ \\ (k, d)=1} \bar{V\Big(\frac{\ell+kn_0/d}X\Big)} \\
= {}& \1_{(d, n_0/d)=1} \Big\{\frac{\vphi(d)}{n_0}|X{\hat V}(0)|^2 + O_{\ee, D}(d^\ee X)\Big\}.
\end{aligned}
\end{equation}
Notons que la propriété~$\rhoQ(p^\nu)=\rhoQ(p)\in\{0, 2\}$ (pour~$p\nmid 2D$, $\nu\geq1$) implique
\begin{equation}
\rhoQ([n_1, n_2])\ssum{d|(n_1, n_2) \\ (d, (n_1, n_2)/d)= 1} 1 = \rhoQ([n_1, n_2])2^{\omega((n_1, n_2))} = \rhoQ(n_1)\rhoQ(n_2).\label{eq:rel-rhon1n2}
\end{equation}
Nous insérons l'estimation~\eqref{eq:somme-l1l2} avec~$X=x/(Mt)$ dans le membre de droite de~\eqref{eq:expr-P1-l1l2} (nous rappelons que l'hypothèse supplémentaire~$M\leq x$ a été justifiée à la section~\ref{sec:premiere-reduction}). Les facteurs~$\vphi(d)$ se compensent, et la relation~\eqref{eq:rel-rhon1n2} nous permet de déduire~$P_1 = P_1' + O(x^{1+\ee})$, avec
$$ P_1' = \frac{|x{\hat V}(0)|^2}{M} C_D \ssumm{N < n_1, n_2 \leq 2N} b_{n_1}\bar{b_{n_2}} \frac{\rhoQ(n_1)\rhoQ(n_2)}{n_1n_2} A(n_1n_2) \int_\RR t^{-2} f(t)\dd t .$$
Nous avons donc~$P_1' = |x {\hat V}(0)|^2 P_3$, et finalement
\begin{equation}
S_1 = |x{\hat V}(0)|^2 P_3 + O_{\ee, D}\Big(x^{1+\ee} + x^{2+\ee} \Big(\frac{N^2}M\Big)^{\frac32-\theta}\Big).\label{eq:estim-S1}
\end{equation}
En insérant les estimations~\eqref{eq:estim-S3}, \eqref{eq:estim-S2}, et~\eqref{eq:estim-S1} dans~\eqref{eq:dispersion}, nous obtenons l'estimation annoncée~\eqref{eq:equidistrib}. Cela conclut la preuve de la Proposition~\ref{prop:equidistrib}

\subsection{Démonstration du \texorpdfstring{Théorème~\ref{th:DI-P}}{Théorème DI-P}}\label{sec:DI-P}

Nous déduisons dans cette section le Théorème~\ref{th:DI-P}, à partir de la majoration~\eqref{eq:somme-expo-moyqh}. Nous suivons les arguments et les notations des sections~4 et~5 de~\cite{DI-Pplus}. Nous considérons
$$ R_H(x, P, D) = \sum_{D<d\leq 2D} \lambda_d \sum_{0<|h|\leq H} \sum_{m\equiv 0\mod{d}} \frac{C(m)\log m}{m} \sum_{\nu^2 \equiv D \mod{m}}{\hat b}\Big(\frac hm\Big) \e\Big(-\frac{h\nu}{m}\Big), $$
où~$D\leq x^{\frac12}$, $P\in[x, x^2]$, $\eta>0$ est arbitraire, $H=Px^{-1+\eta}$, $b$ est une fonction lisse à support compact inclus dans~$[x, 2x]$, telle que~$\|b^{(j)}\|_{\infty}\ll_j x^{-j}$, $C$ est une fonction lisse à support compact inclus dans~$[P, 4P]$, telle que~$\|C^{(j)}\|_{\infty}\ll P^{-j}$, et~$(\lambda_d)$ une suite de coefficients avec~$|\lambda_d|\leq 1$. Nous insérons la définition
$$ \frac1m{\hat b}\Big(\frac hm\Big) = \int_{\RR} \e(-ht) b(mt)\dd t. $$
Posant~$M = P/D$ et~$f_{d, t}(v) = C(Mvd)\log(Mvd)b(Mvdt)$, nous obtenons
$$ |R_H(x, P, D)| \ll xP^{-1} \sup_{|t|\in[x/(4P), 2x/P]} \sum_{D<d\leq 2D} \Big| \sum_{0<|h|\leq H} \e(th) \sum_m f_{d, t}\Big(\frac{m}{M}\Big) \sum_{\nu^2 \equiv D \mod{m}} \e\Big(-\frac{h\nu}{md}\Big)\Big|. $$
Nous avons bien~$\|f_{d, t}^{(j)}\|_\infty \ll_j 1$, $D\ll M$ et $H\ll MD$. Nous sommes donc en mesure d'appliquer la majoration~\eqref{eq:somme-expo-moyqh} à chaque sous-somme dyadique~$H_1<h\leq 2H_1$, pour~$\ud\leq H_1 \leq H$. Nous obtenons
\begin{align*}
R_H(x, P, D) {}& \ll x^{1+\ee+O(\eta)}P^{-1}D \sup_{\ud\leq H_1 \leq H} H_1 \big\{H_1 + M^{\ud} + H_1^{-\ud} D^{\ud-\theta} M^{\ud+\theta}\big\} \\
{}& \ll x^{\ee+O(\eta)}\big\{x^{-1}DP + (DP)^{\ud} + x^{\ud} P^\theta D^{1-2\theta}\big\}.
\end{align*}
Cela est~$O(x^{1-\eta})$ si~$D\leq x^{-K\eta} \min\{x^2P^{-1}, x^{1/(2-4\theta)}P^{-\theta/(1-2\theta)}\}$ et~$K$ est une constante suffisamment grande. Cette majoration de~$D$, en conjonction avec les arguments de la section 8 de~\cite{DI-Pplus}, fournit le résultat annoncé.

\subsection{Démonstration du \texorpdfstring{Lemme~\ref{lemme:pointe}}{Lemme pointe}}\label{sec:demo-pointes}

\begin{proof}
Écrivons~$\sigma \equiv \left(\begin{smallmatrix} u & \ast \\ v & r \end{smallmatrix}\right)$ avec~$r\in\ZZ$. Les classes~$\Gamma_0(qd)\backslash\Gamma$ sont en bijection avec~$\PP^1(\ZZ/qd\ZZ)$, la correspondance étant donnée par~$\sigma\mapsto [v:r]$. La condition~$q|C(\sigma Q)$ correspond alors à~$q|Q(v, r)$.

La relation~$q|C(\sigma Q) = Q(v, r)$ implique~$v | Q(v, r)^2$. Cependant, on a la relation~$Q(v, r)\equiv Q(0, 1)r^2 \mod{v}$ et~$(r, v)=1$, de sorte que finalement~$v|Q(0, 1)^2$.

Le calcul explicite de~$\cC(\infty, \ca)$ et de~$S_{\infty\ca}(h, n ; \gamma)$ est un calcul élémentaire similaire à la section 2.2 de Deshouillers-Iwaniec~\cite{DI}. Nous omettons les détails. La majoration~\eqref{eq:majotriv-kloosterman} en découle par l'inégalité triangulaire, et en notant que la condition~$\alpha\delta \equiv u\mod{vm}$ détermine~$\alpha\mod{vm/(u, m)}$.

Pour la preuve de~\eqref{eq:majogauss-kloosterman}, nous utilisons le théorème des restes chinois. Soit~$p$ un nombre premier, et notons
$$ p^\mu \| m, \quad p^\lambda \| u, \quad p^\nu \| v \quad p^{\nu'}\| v', \quad p^\Delta \| d. $$
Nos hypothèses~$(v, u)=(v', m)=1$ impliquent donc
$$ \mu>0\Rightarrow \nu'=0, \quad \lambda>0\Rightarrow \nu=0, \quad \Delta \leq \max\{\nu, \nu'\}. $$
Le théorème des restes chinois montre que l'on peut se ramener à établir la majoration
\begin{equation}
S_p(h) := \ssum{\alpha \mod{p^{\nu+\mu}} \\ \delta \mod{p^{\lambda+\mu+\max\{\nu, \nu'\}}} \\ \delta\equiv m\mod{p^{\lambda+\nu'}} \\ (\delta-m, p^{\lambda+\mu}) = p^\lambda \\ \alpha\delta \equiv u\mod{p^{\nu+\mu}}} \chi_p(\delta) \e\Big(\frac{h\alpha}{p^{\nu+\mu}}\Big) \ll_Q (p^\Delta h, p^\mu),\label{eq:majo-Sgauss-primpow}
\end{equation}
où~$\chi_p$ est un caractère modulo~$p^\Delta$. Le changement de variables~$\delta\gets m+\delta p^{\lambda+\nu'}$ transforme le membre de gauche en
$$ S_p(h) = \ssum{\alpha \mod{p^{\nu+\mu}} \\ \delta \mod{p^{\mu+\max\{\nu-\nu', 0\}}} \\ (\delta, p^{\mu}) = 1 \\ \alpha(m+\delta p^{\lambda+\nu'}) \equiv u\mod{p^{\nu+\mu}}} \chi_p(m+\delta p^{\lambda+\nu'}) \e\Big(\frac{h\alpha}{p^{\nu+\mu}}\Big). $$

Nous écartons d'abord le cas~$\mu\leq\max\{\lambda, \nu\}$, en nous contentant la borne triviale
$$ S_p(h) \ll_Q 1, $$
qui découle du fait que~$u, v \ll_Q 1$.

Supposons ensuite~$\mu>\max\{\lambda, \nu\}\geq 0$, en particulier,~$\nu'=0$. Considérons d'abord le cas~$\nu=0$. Dans ce cas~$\Delta = 0$, donc le caractère est trivial et notre somme se simplifie en
$$ S_p(h) = \ssum{\alpha\mod{p^\mu} \\ (\alpha, p)=1} \e\Big(\frac{h\alpha}{p^\mu}\Big) \ssum{\delta\mod{p^\mu} \\ \alpha\delta\equiv u/p^\lambda \mod{p^{\mu-\lambda}}}1 = p^\lambda c_{p^\mu}(h), $$
où~$c_r(h) = \sum_{b\mod{r}, (b, r)=1}\e(hb/r)$ est la somme de Ramanujan. Nous obtenons donc
$$ |S_p(h)| \leq p^\lambda (h, p^\mu). $$
Considérons ensuite le cas~$\nu>0$. Cela implique~$\lambda=0$ et~$\Delta\leq \nu$, et donc
$$ S_p(h) = \ssum{\alpha\mod{p^{\nu+\mu}} \\ \delta\mod{p^{\nu+\mu}} \\ (\delta, p)=1 \\ \alpha(m+\delta)\equiv u\mod{p^{\nu+\mu}}} \chi(m+\delta)\e\Big(\frac{h\alpha}{p^{\nu+\mu}}\Big) = \chi(u)\ssum{\alpha\mod{p^{\nu+\mu}}\\(\alpha, p)=1} \chi(\bar{\alpha})\e\Big(\frac{h\alpha}{p^{\nu+\mu}}\Big) $$
qui est une somme de Gauss (\textit{c.f.}~\cite[lemme~3.2]{IK}). Nous avons ainsi
$$ |S_p(h)| \leq 2(p^\Delta h, p^{\nu+\mu}). $$
Nous obtenons dans tous les cas de figure la borne~\eqref{eq:majo-Sgauss-primpow}, ce qui conclut la preuve.
\end{proof}

\providecommand{\bysame}{\leavevmode\hbox to3em{\hrulefill}\thinspace}
\providecommand{\MR}{\relax\ifhmode\unskip\space\fi MR }
\providecommand{\MRhref}[2]{%
  \href{http://www.ams.org/mathscinet-getitem?mr=#1}{#2}
}
\providecommand{\href}[2]{#2}

\end{document}